\documentclass{amsart}

\usepackage[T1]{fontenc}
\usepackage[utf8]{inputenc}

\usepackage[colorlinks,citecolor=blue,urlcolor=blue]{hyperref}

\usepackage[backend=bibtex,doi=true,isbn=false,url=false,
            style=alphabetic,maxbibnames=99,maxcitenames=99,
            maxalphanames=99]{biblatex}
\usepackage{amsthm,amsmath,amsfonts,amssymb}
\usepackage{microtype}
\usepackage{mathtools}
\usepackage{booktabs}
\usepackage{longtable}
\usepackage{multirow}
\usepackage[config, labelfont={normalsize}]{caption,subfig}
\usepackage[backgroundcolor=yellow]{todonotes}
\usetikzlibrary{decorations.pathreplacing}
\usetikzlibrary{decorations.pathmorphing}
\usetikzlibrary{decorations.markings}
\usetikzlibrary{colorbrewer}
\usetikzlibrary{arrows.meta}
\usetikzlibrary{patterns}
\usepackage{subfiles}

\colorlet{colA}{Paired-B}       
\colorlet{colB}{Paired-F}       
\colorlet{colC}{Paired-D}       
\colorlet{colD}{Paired-H}       

\colorlet{colOrange}{Paired-H}  
\colorlet{colBC}{colOrange}     
\colorlet{colP}{Paired-J}       

\colorlet{colGone}{Greys-G}     
\colorlet{colGtwo}{Greys-G}     

\tikzset{
  particleA/.style={colA, line width=1.5pt, solid},
  particleB/.style={colB, line width=0.8pt, double,
    double distance=2pt},
  particleC/.style={colC, line width=1.5pt, decorate,
    decoration={zigzag, segment length=4pt, amplitude=1pt}},
  particleD/.style={colD, line width=2pt, solid,
    postaction={decorate, decoration={markings,
      mark=between positions 2mm and 1 step 5mm
        with {\draw[-] (0,-3pt) -- (0,3pt);}}}},
}

\makeatletter
\@ifclasswith{amsart}{externalize}%
{
  \usetikzlibrary{external}
  \tikzexternalize[
    optimize=false,
    prefix=tikz/,
    mode=list and make]
  \renewcommand{\todo}[2][]{\tikzexternaldisable\@todo[#1]{#2}\tikzexternalenable}
}%
{
  \newcommand{\tikzexternaldisable}{}
  \newcommand{\tikzexternalenable}{}
}
\makeatother

\usepackage{aliascnt}
\usepackage{enumitem}
\usepackage[capitalize,noabbrev]{cleveref}

\numberwithin{equation}{section}

\DeclareMathOperator{\id}{id}
\DeclareMathOperator{\Pf}{Pf}

\theoremstyle{definition}
\newtheorem{definition}{Definition}[section]

\theoremstyle{plain}
\newaliascnt{theorem}{definition}
\newtheorem{theorem}[theorem]{Theorem}
\aliascntresetthe{theorem}

\newaliascnt{proposition}{definition}
\newtheorem{proposition}[proposition]{Proposition}
\aliascntresetthe{proposition}

\newaliascnt{lemma}{definition}
\newtheorem{lemma}[lemma]{Lemma}
\aliascntresetthe{lemma}

\newaliascnt{corollary}{definition}
\newtheorem{corollary}[corollary]{Corollary}
\aliascntresetthe{corollary}

\theoremstyle{remark}
\newaliascnt{example}{definition}
\newtheorem{example}[example]{Example}
\aliascntresetthe{example}

\newaliascnt{remark}{definition}
\newtheorem{remark}[remark]{Remark}
\aliascntresetthe{remark}

\newaliascnt{principle}{definition}
\newtheorem{principle}[principle]{Principle}
\aliascntresetthe{principle}

\crefname{theorem}{Theorem}{Theorems}
\Crefname{theorem}{Theorem}{Theorems}
\crefname{lemma}{Lemma}{Lemmas}
\Crefname{lemma}{Lemma}{Lemmas}
\crefname{corollary}{Corollary}{Corollaries}
\Crefname{corollary}{Corollary}{Corollaries}
\crefname{proposition}{Proposition}{Propositions}
\Crefname{proposition}{Proposition}{Propositions}
\crefname{definition}{Definition}{Definitions}
\Crefname{definition}{Definition}{Definitions}
\crefname{example}{Example}{Examples}
\Crefname{example}{Example}{Examples}
\crefname{remark}{Remark}{Remarks}
\Crefname{remark}{Remark}{Remarks}
\crefname{principle}{Principle}{Principles}
\Crefname{principle}{Principle}{Principles}

\usepackage{needspace}
\let\origsubsection\subsection
\renewcommand{\subsection}{\needspace{6\baselineskip}\origsubsection}
\let\origsubsubsection\subsubsection
\renewcommand{\subsubsection}{\needspace{3\baselineskip}\origsubsubsection}

\newcommand{\ZZ}{\mathbb{Z}}
\newcommand{\RR}{\mathbb{R}}
\newcommand{\PP}{\mathbb{P}}

\DeclareMathOperator{\sgn}{sgn}

\newcommand{\Actors}{\mathcal{A}}
\newcommand{\Roles}{\mathcal{R}}
\newcommand{\Ghosts}{\mathcal{G}}
\newcommand{\Survivors}{\mathcal{S}}
\newcommand{\FinalState}{\mathcal{F}}

\newcommand{\casting}{\mathcal{C}}
\newcommand{\perf}{\mathcal{P}}
\newcommand{\paths}{\mathbf{P}}

\newcommand{\Attribution}{\mathfrak{A}}

\newcommand{\Candidates}{\Pi_{\FinalState}}

\title[Particle annihilation]%
{Determinant and Pfaffian formulas \\ for particle annihilation}
\author{Piotr \'Sniady}
\address{Institute of Mathematics, Polish Academy of Sciences,
         ul.~\'Sniadeckich 8, \mbox{00-656~Warszawa,} Poland}
\email{psniady@impan.pl}

\begin{document}

\begin{abstract}
We consider systems of particles on a line in which
colliding particles annihilate each other and vanish.
Computing exact annihilation probabilities
is difficult because every collision reduces the particle
count, while determinantal methods require a fixed count
throughout. The ghost particle method, introduced in a
companion paper for coalescence, removes the obstacle:
destroyed particles continue walking as invisible ghosts, so
the number of trajectories never changes. Applied to
annihilation, the method yields an exact determinantal formula
for the probability of any prescribed outcome---the number of
annihilations, the survivor positions, and the positions of
the ghosts. For complete annihilation, where no particle
survives, the determinant collapses to a Pfaffian, an
algebraic relative of the determinant built from pairwise
quantities: although the particles interact, the extinction
probability is determined by pairwise annihilation
probabilities alone. This gives a combinatorial explanation of
the Pfaffian structure of annihilating systems, previously
derived through differential equations for specific dynamics.
The annihilation formula also yields results about
coalescence: the event that prescribed pairs of particles have
merged can be reinterpreted as complete annihilation,
producing a Pfaffian coalescence formula. All formulas are
exact for any finite initial configuration and apply to
discrete lattice paths, birth-death chains, and continuous
diffusions including Brownian motion.
\end{abstract}



\maketitle

\vspace{1em}
\begin{quote}
\textit{To Marek Bożejko, \\
who taught me that in creation and annihilation,\\
order is everything.}
\end{quote}
\vspace{1em}


\section{Introduction}\label{sec:intro}

\subsection{The problem}
\label{sec:intro-problem}

Consider $n$ particles performing independent random walks
on a one-dimensional universe for which a version of the
\emph{`Darboux property'} holds true: two particles cannot
swap order without being at some intermediate moment in the
same state. (We postpone the formal definition to the
crossing property of \cref{def:planar}; a concrete example
to keep in mind is a system of simple random walks, each
taking $\pm 1$ steps and starting at even integers, see 
\cref{fig:intro-annihilation}.)
When two particles meet, they annihilate: both are destroyed. This
$A + A \to \emptyset$ reaction-diffusion model appears throughout
statistical physics. In the one-dimensional Ising--Glauber
model~\cite{Glauber1963}, a spin chain evolves by single-spin
flips, each occurring at a rate determined by the neighboring
spins; at zero temperature the domain walls---boundaries
between spin-up and spin-down regions---perform random walks
on the dual lattice, and when two walls meet, the intervening
region disappears and both walls are destroyed. The same
dynamics governs models of diffusion-limited chemical kinetics
in low dimensions~\cite{benAvrahamHavlin2000}.
\emph{What is the probability that exactly $k$ annihilations
occur, with survivors reaching specified positions?}

For \emph{non-colliding} particles, exact probabilities are
classical. The Karlin--McGregor theorem~\cite{KM1959} and its
combinatorial cousin, the Lindstr\"om--Gessel--Viennot (LGV)
lemma~\cite{Lindstrom1973,GV1985}, express the probability
that $n$ particles starting at positions
$x_1 \leq \cdots \leq x_n$ reach positions
$y_1 \leq \cdots \leq y_n$ without colliding as a determinant:
\[
\PP(\text{particles reach } y_1, \ldots, y_n
  \text{ without colliding})
= \det \bigl( p(x_i \to y_j) \bigr)_{1 \leq i,j \leq n},
\]
where $p(x_i \to y_j)$ is the transition probability from $x_i$
to $y_j$.

When particles annihilate, the count decreases. After $k$ collisions,
only $n - 2k$ particles remain. The number of rows (initial particles)
exceeds the number of columns (final particles), and there is no
square matrix to write down. This dimensional mismatch places
annihilating systems outside the scope of the classical LGV framework.

\subsection{The ghost method}
\label{sec:intro-ghost}

The companion paper~\cite{SU2026coalescence} introduced
\emph{ghost particles} for coalescence: when two particles merge,
one heir and one ghost emerge, keeping the entity count at~$n$
and enabling determinantal formulas.

For annihilation, the same idea applies with a twist: when two
particles annihilate, an ordered pair of ghosts emerges from the
collision point---invisible walkers that perform independent
random walks onward, without interacting with anything.
The ghosts are anonymous: only the pairing of the
two ghosts born together is
preserved (\Cref{sec:setup-annihilation};
see also \Cref{sec:intro-scope-structure}).
The entity
count---survivors plus ghosts---remains exactly~$n$, restoring
the dimensional structure needed for determinantal methods.
We call this variant the \emph{ghost pair method}.
\Cref{fig:intro-annihilation} illustrates: particles~$2$
and~$3$ annihilate at~$c$, and a ghost pair emerges, while
particles~$1$ and~$4$ survive.

At a structural level, the ghost pair method is a
minimal extension of the Karlin--McGregor theorem to
annihilation: the same sign-reversing involution
applied to the same determinantal language, with one
new ingredient---ghost pairs that restore the matrix
to square form.


\begin{figure}[t]
\centering

\tikzset{
  ghostPairA/.style={colGone, line width=1.5pt, dashed,
    dash pattern={on 4pt off 2pt}},
  ghostPairB/.style={colGone, line width=1.5pt, dashed,
    dash pattern={on 4pt off 2pt}, double, double distance=2pt},
}

\begin{tikzpicture}[scale=0.55]
  \begin{scope}
    \clip (-0.2,-0.2) rectangle (11.2,6.2);

    \foreach \x in {-1,...,12} {
      \foreach \t in {-1,...,7} {
        \pgfmathparse{mod(\x+\t,2)==0 ? 1 : 0}
        \ifnum\pgfmathresult>0
          \fill[gray!30] (\x,\t) circle (1.5pt);
        \fi
      }
    }

    \foreach \x in {-2,0,2,4,6,8,10,12} {
      \foreach \t in {-2,0,2,4,6,8} {
        \draw[gray!30, thin] (\x,\t) -- (\x-1,\t+1);
        \draw[gray!30, thin] (\x,\t) -- (\x+1,\t+1);
      }
    }
    \foreach \x in {-1,1,3,5,7,9,11} {
      \foreach \t in {-1,1,3,5,7} {
        \draw[gray!30, thin] (\x,\t) -- (\x-1,\t+1);
        \draw[gray!30, thin] (\x,\t) -- (\x+1,\t+1);
      }
    }
  \end{scope}

  \draw[->, thick] (-0.3,0) -- (11,0) node[right] {$x$};
  \draw[->, thick] (0,-0.3) -- (0,6.7) node[above] {$t$};

  \draw (-0.1, 6) -- (0.1, 6);
  \node[left] at (-0.15,6) {\small $T$};

  \fill[colA] (4,0) circle (4pt);
  \fill[colB] (6,0) circle (4pt);
  \fill[colC] (8,0) circle (4pt);
  \fill[colD] (10,0) circle (4pt);

  \node[below] at (4,-0.15) {\small $x_1$};
  \node[below] at (6,-0.15) {\small $x_2$};
  \node[below] at (8,-0.15) {\small $x_3$};
  \node[below] at (10,-0.15) {\small $x_4$};


  \draw[ghostPairA] (6,2) -- (5,3);
  \draw[ghostPairA, transform canvas={shift={(0.12,0.12)}}]
    (5,3) -- (4,4) -- (3,5);
  \draw[ghostPairA] (3,5) -- (2,6);

  \draw[ghostPairB] (6,2) -- (7,3);
  \draw[ghostPairB, transform canvas={shift={(0.1,-0.1)}}]
    (7,3) -- (8,4) -- (9,5);
  \draw[ghostPairB] (9,5) -- (8,6);


  \draw[particleA] (4,0) -- (3,1) -- (4,2) -- (5,3) -- (4,4)
    -- (3,5) -- (4,6);

  \draw[particleB] (6,0) -- (5,1) -- (6,2);

  \draw[particleC] (8,0) -- (7,1) -- (6,2);

  \draw[particleD] (10,0) -- (9,1) -- (8,2) -- (7,3) -- (8,4)
    -- (9,5) -- (10,6);

  \fill[black] (6,2) circle (5pt);
  \node[left=0.2] at (6,2) {\small $c$};

  \fill[colA] (4,6) circle (4pt);                                    
  \fill[colD] (10,6) circle (4pt);                                   
  \draw[colGone, fill=white, line width=1.5pt] (2,6) circle (4pt);   
  \draw[colGone, fill=white, line width=1.5pt] (8,6) circle (4pt);   

  \node[above] at (2,6.15) {\small $a$};
  \node[above] at (4,6.15) {\small $y_1$};
  \node[above] at (8,6.15) {\small $b$};
  \node[above] at (10,6.15) {\small $y_2$};

  \begin{scope}[shift={(11,0.8)}]
    \draw[particleA] (0,5.2) -- (0.7,5.2);
    \node[right] at (0.8,5.2) {\small particle $1$ (survivor)};
    \draw[particleB] (0,4.4) -- (0.7,4.4);
    \node[right] at (0.8,4.4) {\small particle $2$};
    \draw[particleC] (0,3.6) -- (0.7,3.6);
    \node[right] at (0.8,3.6) {\small particle $3$};
    \draw[ghostPairA] (0,2.4) -- (0.7,2.4);
    \draw[ghostPairB] (0,1.8) -- (0.7,1.8);
    \node[right] at (0.8,2.1) {\small ghost pair};
    \draw[particleD] (0,0.6) -- (0.7,0.6);
    \node[right] at (0.8,0.6) {\small particle $4$ (survivor)};
  \end{scope}

\end{tikzpicture}

\caption{\textbf{Annihilation on the checkerboard lattice.}
Four particles start at $x_1 < x_2 < x_3 < x_4$. Particles~$2$
(double) and~$3$ (zigzag) annihilate at~$c$; both are destroyed
and an ordered pair of ghosts emerges (dashed paths). Particles~$1$ (solid)
and~$4$ (tick marks) survive. Ghost paths freely cross survivor paths
(shown offset)---ghosts do not interact. Final positions:
$a < y_1 < b < y_2$.}
\label{fig:intro-annihilation}

\end{figure}

\subsection{Main results}
\label{sec:intro-results}

\subsubsection{The annihilation formula}

Consider $n$ particles starting at $x_1 \leq \cdots \leq x_n$.
Suppose $k$ collisions occur, producing $s = n - 2k$ survivors
at positions $y_1 \leq \cdots \leq y_s$ (weakly increasing)
and $k$ ghost pairs. The
ghost pairs are physically indistinguishable---they do not
remember which particles created them---so we label them
$1, \ldots, k$ by sampling a uniform random numbering. Ghost
pair~$j$ then has positions $(a_j, b_j)$.

\begin{theorem}[Annihilation formula, preview]
\label{thm:intro-annihilation}
For a fixed final state---survivor positions $y_1, \ldots, y_s$
and ghost pair positions $(a_1, b_1), \ldots, (a_k, b_k)$---the
probability is:
\[
\PP = \frac{1}{k!} \det(M),
\]
where $M$ is the $n \times n$ matrix described below.
\end{theorem}

The factor $1/k!$ is the probability of the specific numbering
$1, \ldots, k$; the body of the paper works with weights rather
than probabilities, and there the numbering-free statement reads
$Z = \det(M)$ (\Cref{thm:annihilation}).

The matrix $M$ has rows indexed by initial particles and
columns indexed by final entities (survivors and ghosts). Survivor
columns contain transition probabilities $p(x_i \to y_\ell)$. Ghost
columns additionally carry formal variables $\tau^{(j)\pm}_I$,
indexed by the ghost pair~$j$ and the row~$I$; their products
evaluate to $0$ or $\pm 1$, depending on whether the two
particles assigned to ghost pair~$j$ are ordered consistently
with the \emph{ghost sign} $\varepsilon_j$ ($+1$ if
$a_j \leq b_j$, that is, the ghost pair is ordered left-right;
$-1$ if $a_j > b_j$). The precise
statement, with transition probabilities generalized to path
generating functions, is \Cref{thm:annihilation}.

\subsubsection{Example: four particles, one collision}

Four particles start at $x_1 \leq x_2 \leq x_3 \leq x_4$,
with one collision ($k = 1$). Two survivors end at
$y_1 \leq y_2$ and one ghost pair at positions $a$ and $b$
(as in \Cref{fig:intro-annihilation}).

The $4 \times 4$ matrix $M$ has two survivor columns and two
ghost columns (one pair). Since $k = 1$, we drop the pair
index, writing $\tau^\pm_I$ for $\tau^{(1)\pm}_I$:
\[
M = \begin{pmatrix}
p(x_1 \to y_1) & p(x_1 \to y_2) &
  \tau^+_1 p(x_1 \to a) & \tau^-_1 p(x_1 \to b) \\
p(x_2 \to y_1) & p(x_2 \to y_2) &
  \tau^+_2 p(x_2 \to a) & \tau^-_2 p(x_2 \to b) \\
p(x_3 \to y_1) & p(x_3 \to y_2) &
  \tau^+_3 p(x_3 \to a) & \tau^-_3 p(x_3 \to b) \\
p(x_4 \to y_1) & p(x_4 \to y_2) &
  \tau^+_4 p(x_4 \to a) & \tau^-_4 p(x_4 \to b)
\end{pmatrix}.
\]
The first two columns are plain transition probabilities (survivor
columns). The last two columns carry additional formal variables
$\tau^\pm_I$ indexed by the row. Writing $\varepsilon$ for the
ghost sign ($+1$ if $a \leq b$, $-1$ if $a > b$), these
variables obey the multiplication rule
\[
\tau^+_I \, \tau^-_J = \begin{cases}
-1 & \text{if } \varepsilon = +1 \text{ and } I > J, \\
+1 & \text{if } \varepsilon = -1 \text{ and } I < J, \\
\phantom{+}0 & \text{otherwise},
\end{cases}
\]
so the only surviving terms in $\det(M)$ are those in which
the higher-indexed of the two ghost actors is sent to the
leftward ghost position.

To compute the annihilation probability, expand $\det(M)$ by
Laplace expansion along the two ghost columns. Each term
assigns a pair of particles $\{I, J\}$ with $I < J$ to the
ghost pair and the remaining two to the survivors. Working
out the multiplication rule gives, for either sign
of~$\varepsilon$, a single formula:
\begin{multline}\label{eq:intro-laplace}
\PP = \frac{1}{1!} \sum_{\substack{I < J \\
  \{I,J\} \subset \{1,2,3,4\}}}
(-1)^{I+J+1}\,
p\bigl(x_I \to \max(a,b)\bigr)\,
p\bigl(x_J \to \min(a,b)\bigr) \\
\cdot \det \bigl( p(x_i \to y_\ell) \bigr)_{
  \substack{i \in \{1,\ldots,4\} \setminus \{I,J\} \\
  \ell \in \{1,2\}}},
\end{multline}
where the sign $(-1)^{I+J+1}$ comes from the Laplace
expansion along columns~$3$ and~$4$, and each $2 \times 2$
determinant is a Karlin--McGregor probability for the two
surviving particles. The higher-indexed particle~$J$ ends at $\min(a,b)$
and the lower-indexed particle~$I$ at $\max(a,b)$: the
index ordering is reversed relative to the ghost pair's
spatial ordering.

\subsubsection{The Pfaffian connection}

For complete annihilation (all particles destroyed), the
formula simplifies once we no longer prescribe the final
positions of the ghosts. Summing over them, the probability
that $n = 2k$ particles annihilate completely is a Pfaffian
of pairwise quantities:
\[
\PP = \Pf(A),
\]
where $A_{IJ}$ measures the total weight of crossing paths
for particles~$I$ and~$J$---equivalently, the probability
that $I$ and~$J$ would annihilate if they were the only two
particles. For $n = 4$ particles the Pfaffian has three
summands, one per perfect matching of $\{1, 2, 3, 4\}$:
\[
\PP = A_{12}\, A_{34} - A_{13}\, A_{24} + A_{14}\, A_{23};
\]
the negative term corresponds to the matching
$\{1,3\}, \{2,4\}$, which no collision sequence can realize;
its role is to cancel the spurious contributions hidden in
the other two products (\Cref{ex:pfaffian-n4} works out the
details).

Why a Pfaffian? Coalescing and annihilating particle
systems are known to be governed by \emph{Pfaffian} point
processes~\cite{TZ2011,GarrodPTZ2018}
(\Cref{sec:intro-prior}). For coalescence this is at first
sight puzzling: nothing in the merging rule suggests
pairings. Annihilation is where the pairing is visible.
Complete annihilation of $2k$ particles requires them to
pair up, each pair meeting and destroying both members, and
the ways to pair $2k$ particles are the perfect matchings
of $\{1, \ldots, 2k\}$---exactly the combinatorial objects
that define a Pfaffian. The ghost pair structure makes this
correspondence precise: the Leibniz expansion of the
determinant groups its terms by matching, with each
matching assigning particles to ghost pairs. Coalescence,
by contrast, involves compositions (how many particles
merge into each heir) and inherits the Pfaffian structure
only through its reduction to annihilation: the
\emph{cancellative labeling} of \Cref{sec:pfaffian}
converts pairwise coalescence into complete annihilation.
The companion paper~\cite{Sniady2026pfaffian} develops the
consequences for coalescing particle systems in full:
there, empty intervals of a coalescing system are exactly
complete-annihilation events, which is how the Pfaffian
point processes above arise.

\subsubsection{Relation to the LGV lemma}
\label{sec:intro-lgv}

The annihilation formula generalizes the
Karlin--McGregor / Lindstr\"om--Gessel--Viennot (LGV)
determinant, just as the coalescence formula
does~\cite{SU2026coalescence}. In both cases, the proof
uses a sign-reversing involution that swaps path segments at
the first wrong crossing---but in annihilation, ``wrong''
means ``crossing that violates the prescribed collision
pattern'' rather than any crossing at all. When no collisions
occur (all particles survive, no ghosts), the formula reduces
to the standard LGV determinant.
See~\cite{SU2026coalescence} for the full discussion of
the LGV connection, including the relationship to
Stembridge's $D$-compatibility~\cite{Stembridge1990}.

\subsection{Prior work}
\label{sec:intro-prior}

\subsubsection{Pfaffian point processes}
\label{sec:pfaffian-point-processes}

Tribe and Zaboronski~\cite{TZ2011} proved Pfaffian
point process structure for coalescing and
annihilating Brownian
motions under the maximal entrance law: their
Pfaffian structure (Theorem~2 in~\cite{TZ2011}) is
exact, while their $n$-point density formula
(Theorem~1) is asymptotic for large~$t$. Garrod,
Poplavskyi, Tribe, and
Zaboronski~\cite{GarrodPTZ2018} extended this to
continuous-time random walks on~$\ZZ$ with spatially
inhomogeneous rates and all deterministic initial
conditions, covering mixed
coalescence-annihilation ($\theta \in [0,1]$). Tribe
and Zaboronski~\cite{TribeZaboronski2026} further
extended this to all entrance laws. These proofs
rest on a time-homogeneous Markov generator and the
\emph{spin-pair identity}: expectations of products
of pairwise spin variables satisfy a closed system
of ordinary differential equations (ODE), which the
candidate Pfaffian also solves, with the same
initial condition; the two agree by uniqueness.
Conversely, they reach results that the
ghost method does not: mixed
coalescence-annihilation with arbitrary~$\theta$, and
the classification of all entrance laws.

In an earlier, separate direction,
Mattera~\cite{Mattera2003} obtained a
Kasteleyn--Pfaffian description of annihilating
random walks on one specific lattice;
\Cref{sec:pfaffian-discussion} discusses the
relationship.

\subsubsection{The IPDF method}

The interparticle distribution function (IPDF) method
computes density and correlation functions by tracking
the probability distribution of interparticle gaps.
Doering and
ben-Avraham~\cite{DoeringBenAvraham1988} introduced
it for diffusion-limited coalescence;
ben-Avraham~\cite{benAvraham1998} extended it to
the full hierarchy of $n$-point correlation
functions; Masser and
ben-Avraham~\cite{MasserBenAvraham2001} adapted it
to annihilation via the method of intervals (see
also the monograph of ben-Avraham and
Havlin~\cite{benAvrahamHavlin2000}). Unlike
the generator/ODE approach, the IPDF method does
not require a Markov generator, but relies on
Brownian-motion-specific integrals.

\subsubsection{Cancellative duality}

A parity correspondence converts coalescence into
annihilation: this is the classical cancellative
duality of Griffeath~\cite{Griffeath1979}, which
\Cref{sec:pfaffian} recasts deterministically in the
ghost framework to yield the Pfaffian formula. The
attribution and mechanism are given there.

\subsection{Scope and structure of the method}
\label{sec:intro-scope-structure}

Beyond the specific formulas, the ghost pair method has
several structural features---wide applicability, fine
resolution, and the anonymity of its ghosts---that
distinguish it from the analytic approaches.

\subsubsection{Wide scope}

The ghost pair method is combinatorial, as in the
companion coalescence
paper~\cite{SU2026coalescence}. It requires
only the Karlin--McGregor assumptions: identical,
independent dynamics, order preservation, the
strong Markov property, and the requirement that
meeting times are stopping times~\cite{KM1959}. These hold for any
\emph{skip-free} process---one whose transitions go
only to neighboring states, so that particles cannot
change order without first meeting. No generator, no
PDE, and no spin-pair identity
(\Cref{sec:pfaffian-point-processes}) is needed. The formula
therefore applies uniformly to lattice paths,
birth-death chains, and Brownian motion, including
discrete lattices with arbitrary inhomogeneous
transition probabilities (varying in both space and
time) where no generator is available. The proofs are
given in the discrete spacetime-graph setting
(\Cref{sec:setup}), where the combinatorial core is
cleanest; the continuous statements (Brownian motion,
birth-death chains) are given in \Cref{sec:continuous}
and follow by the measure-theoretic transfer scheme of
the companion paper~\cite{SU2026coalescence}.

\subsubsection{Exact finite-time formulas}

The ghost formula gives exact fixed-time probabilities
for specific outcomes of finite configurations: how many
annihilations occur by time~$T$, where the survivors end
up, and where the ghost pairs end up. This is finer
resolution than the density and correlation functions
provided by the analytic approach, which describe the
statistical structure of infinitely many particles but
not the outcome of a prescribed finite configuration.
For the systems of \Cref{sec:intro-problem}, this means
exact transition probabilities for any finite
configuration of domain walls or reacting particles.

\subsubsection{Ghost anonymity}

Ghost pairs do not remember which initial particles
produced them. Can this be refined to prescribe
\emph{which} particles annihilate? \Cref{sec:prescribed}
provides computational evidence that no such
refinement exists: the prescribed annihilation
probability cannot be expressed as any linear
combination of Karlin--McGregor products, even
allowing arbitrary rational coefficients. This evidence indicates that
ghost anonymity is not a bookkeeping convenience
but a structural necessity.

\subsection{Companion papers}
\label{sec:intro-companion}

This paper stands between two companion papers on
coalescence ($A + A \to A$): it adapts the ghost method
from the first, and it supplies the key combinatorial
input to the second.

The foundational companion
paper~\cite{SU2026coalescence} develops the coalescence
formula, a ghost-free coalescence determinant (integrating
out ghost positions), and the full continuous-time
treatment, whose transfer scheme \Cref{sec:continuous}
reuses. The two papers share the same proof
architecture---sign-reversing involution via segment
swap---but differ in the ghost structure: annihilation
produces ghost pairs with paired formal variables, while
coalescence produces single ghosts with a staircase
pattern.

In the opposite direction, the Pfaffian companion
paper~\cite{Sniady2026pfaffian} applies the Pfaffian
pairwise coalescence formula proved here to the walls of
the coalescing process---the boundaries between the sets
of sites whose particles merge into the same
survivor---obtaining an empty-interval formula, a cumulant
coloring formula, and a central limit theorem for the wall
count. Without the Pfaffian reduction proved here, these
results would have no combinatorial foundation.

\subsection{Organization}
\label{sec:intro-organization}

\Cref{sec:setup,sec:formula} set up the annihilation
model and state the formula.
\Cref{sec:proof} proves it via a sign-reversing
involution. \Cref{sec:continuous} states the continuous
versions of the results. \Cref{sec:pfaffian} derives the
Pfaffian pairwise coalescence formula by converting
coalescence to complete annihilation via the cancellative
labeling. \Cref{sec:prescribed} reports the computational
evidence on necessity of ghost anonymity.

The companion code, which numerically verifies the
annihilation formula (\Cref{thm:annihilation}) and the
Pfaffian formula (\Cref{thm:pfaffian}) and performs
the computations reported in \Cref{sec:prescribed}, is
archived at~\cite{Sniady2026companion}.

\section{Setup}\label{sec:setup}

The annihilation formula holds for random walks on~$\ZZ$,
Brownian motion on~$\RR$, and birth-death chains on arbitrary
state spaces. Following the companion
paper~\cite{SU2026coalescence}, we work with
\emph{spacetime graphs}---a combinatorial abstraction that
captures two structural properties common to all these
settings:
\begin{enumerate}[label=(\roman*)]
\item \textbf{Planarity}: paths with swapped endpoints
  must cross, and non-adjacent particles cannot meet
  without an intermediate particle involved;
\item \textbf{Weight-preserving segment swap}: exchanging
  path segments at a shared vertex preserves the product
  of weights.
\end{enumerate}
For discrete models, the spacetime graph is literal; for
continuous processes, the combinatorial structure is identical
and the results are stated in \Cref{sec:continuous}, with
proofs by the measure-theoretic transfer scheme of the
companion paper~\cite{SU2026coalescence}.

\subsection{Spacetime graphs}
\label{sec:setup-spacetime}

\begin{definition}[Spacetime graph]\label{def:spacetime-graph}
A \emph{spacetime graph} is a directed, acyclic graph
$D = (V, E)$ with edge weights $w\colon E \to R$, where
$R$ is a commutative ring. The acyclicity induces a
\emph{time ordering}: $u \lessdot v$ if there is a
directed path from $u$ to~$v$. This is a partial order;
we fix a linear extension of~$\lessdot$. The phrase
``first crossing'' (used in the proof,
\Cref{sec:proof-rehearsal}) means first in this linear
order; the proof works for any such extension.
\end{definition}

\begin{definition}[Paths and weights]\label{def:path-weight}
A \emph{path} from $x$ to $y$ is a sequence of vertices
$(v_0, \ldots, v_\ell)$ with $v_0 = x$, $v_\ell = y$, and each
$(v_i, v_{i+1}) \in E$. The \emph{weight} of a path is
$w(P) = \prod_{i} w(v_i \to v_{i+1})$.
The \emph{path generating function} is
\[
W(x \to y) = \sum_{P: x \to y} w(P).
\]
\end{definition}

When the edge weights are the one-step transition
probabilities of a random walk, $W(x \to y)$ is the
probability that the walk moves from $x$ to~$y$, and
every weight in this paper---including the total
weight~$Z$ of \Cref{thm:annihilation}---becomes a
probability, as in \Cref{sec:intro}.

\subsection{Planarity}
\label{sec:setup-planarity}

\begin{definition}[Source and target sets]\label{def:source-target}
The \emph{source set} $\mathcal{X} \subseteq V$ and \emph{target set}
$\mathcal{Y} \subseteq V$ are each equipped with a linear order $\prec$.
\end{definition}

\begin{definition}[Planar configuration]\label{def:planar}
The pair $(\mathcal{X}, \mathcal{Y})$ is \emph{planar} if:
\begin{enumerate}[label=(P\arabic*)]
\item \label{item:crossing} \textbf{Crossing property.}
  For $x \prec x'$ in $\mathcal{X}$ and $y' \prec y$ in $\mathcal{Y}$
  (targets swapped), every path from $x$ to $y$ intersects every path
  from $x'$ to $y'$.
\item \label{item:consecutive} \textbf{Consecutive collision property.}
  For $x \prec x' \prec x''$ in $\mathcal{X}$, if paths from $x$ and $x''$
  meet at vertex $v$, then every path from $x'$ must pass through $v$ or
  intersect one of those paths before $v$.
\end{enumerate}
\end{definition}

The crossing property~\ref{item:crossing} is Stembridge's
$D$-compatibility~\cite{Stembridge1990}: paths with swapped
endpoints must meet (see \Cref{fig:planarity-p1}). The
classical LGV lemma needs only this
condition, because it forbids all crossings. The consecutive
collision property~\ref{item:consecutive}, introduced in the
companion paper~\cite{SU2026coalescence}, is needed
because we allow collisions but require them to respect the
spatial ordering: non-adjacent particles cannot collide
without involving intermediate ones
(see \Cref{fig:planarity-p2}). Both properties hold
for lattice paths, random walks on~$\ZZ$, birth-death
chains, and Brownian motion;
see~\cite{SU2026coalescence} for verification. The
consecutive collision property is essential for the
sign-reversing involution: it ensures that when two paths
cross, the corresponding particles are adjacent in the
active set (\Cref{prop:adjacent-crossings}).

\tikzset{
  nosign radius/.initial=0.38,
  nosign/.pic={
    \draw[red, line width=1.3pt] (0,0)
      circle[radius={\pgfkeysvalueof{/tikz/nosign radius}}];
    \draw[red, line width=1.3pt]
      (-135:{\pgfkeysvalueof{/tikz/nosign radius}})
      -- (45:{\pgfkeysvalueof{/tikz/nosign radius}});
  },
}


\begin{figure}[t]
\centering
\captionsetup[subfigure]{justification=centering}

\subfloat[]{%
\begin{tikzpicture}[scale=0.6,
    vertex/.style={circle, fill, inner sep=1.5pt}]

\draw[gray, thick] (0, 0) -- (6, 0);
\draw[gray, thick] (0, 5) -- (6, 5);
\node[right] at (6.1, 0) {$\mathcal{X}$};
\node[right] at (6.1, 5) {$\mathcal{Y}$};

\node[vertex, colB] (a0) at (1, 0) {};   \node[below] at (a0) {$A$};
\node[vertex, colB] at (5, 5) {};
\node[vertex, colA] (b0) at (5, 0) {};   \node[below] at (b0) {$B$};
\node[vertex, colA] at (1, 5) {};

\draw[particleA] (5, 0) -- (1, 5);
\draw[particleB] (1, 0) -- (2.4, 1.75);
\draw[white, line width=6pt]
  (2.4, 1.75) to[out=95, in=175] (3.0, 2.95) to[out=-5, in=140] (3.6, 3.25);
\draw[particleB]
  (2.4, 1.75) to[out=95, in=175] (3.0, 2.95) to[out=-5, in=140] (3.6, 3.25);
\draw[particleB] (3.6, 3.25) -- (5, 5);

\pic at (4.5, 2.35) {nosign};

\node[font=\scriptsize] at (1.75, 2.95) {bridge};

\end{tikzpicture}%
\label{fig:planarity-p1}%
}
\hfill
\subfloat[]{%
\begin{tikzpicture}[scale=0.6,
    vertex/.style={circle, fill, inner sep=1.5pt},
    collision/.style={circle, fill=black, inner sep=1.8pt}]

\draw[gray, thick] (0, 0) -- (6.4, 0);
\draw[gray, thick] (0, 5) -- (6.4, 5);
\node[right] at (6.5, 0) {$\mathcal{X}$};
\node[right] at (6.5, 5) {$\mathcal{Y}$};

\node[vertex, colA] (a0) at (1, 0) {};   \node[below] at (a0) {$A$};
\node[vertex, colB] (b0) at (3, 0) {};   \node[below] at (b0) {$B$};
\node[vertex, colC] (c0) at (5, 0) {};   \node[below] at (c0) {$C$};

\node[vertex, colA] at (1, 5) {};
\node[vertex, colB] at (3, 5) {};
\node[vertex, colC] at (5, 5) {};

\node[collision] (v) at (3, 2.5) {};
\node[left=2pt] at (v) {$v$};

\node[collision] (w) at (3.95, 3.6) {};
\node[left=2pt] at (w) {$w$};

\draw[particleA] (1, 0) -- (v) -- (1, 5);

\draw[particleC] (5, 0) -- (4.45, 0.69);     
\draw[particleC] (3.95, 1.31) -- (v);        
\draw[particleC] (v) -- (w) -- (5, 5);        

\draw[particleB] (3, 0) -- (3.9, 0.8);       
\draw[white, line width=6pt]
  (3.9, 0.8) to[out=60, in=180] (4.45, 1.5) to[out=0, in=120] (5.05, 1.2);
\draw[particleB]
  (3.9, 0.8) to[out=60, in=180] (4.45, 1.5) to[out=0, in=120] (5.05, 1.2);
\draw[particleB]
  (5.05, 1.2) to[out=40, in=-85] (5.5, 2.5) to[out=95, in=-25] (w);
\draw[particleB] (w) -- (3, 5);              

\pic at (6.25, 1.45) {nosign};

\node[font=\scriptsize] at (4.5, 2.05) {bridge};

\end{tikzpicture}%
\label{fig:planarity-p2}%
}

\caption{\textbf{The two planarity conditions} (\Cref{def:planar}), each shown
as a \emph{forbidden} configuration; time runs upward, with sources in
$\mathcal{X}$ (bottom) and targets in $\mathcal{Y}$ (top).
\textbf{(a)~Crossing property~\ref{item:crossing}.} Paths $A$ and $B$ have
swapped endpoints---$A$ runs from the left source to the right target and $B$
vice versa---yet reach them \emph{without} intersecting, one bridging over the
other. Condition~\ref{item:crossing} forbids this: paths with swapped
endpoints must cross.
\textbf{(b)~Consecutive collision property~\ref{item:consecutive}.} Paths
$A$ and $C$ intersect at the shared vertex $v \in A \cap C$.
The intermediate path~$B$ meets $A\cup C$ only at the vertex marked~$w$,
which lies \emph{above}~$v$.
Condition~\ref{item:consecutive} forbids this: $B$ must share a vertex
with $A$ or~$C$ at or before~$v$.}
\label{fig:planarity}

\end{figure}

\subsection{The annihilation model}
\label{sec:setup-annihilation}

Fix $n$ source vertices
$x_1 \preceq x_2 \preceq \cdots \preceq x_n$ in~$\mathcal{X}$.
The target set $\mathcal{Y}$ collects the possible final
positions---for a walk observed for $T$ steps, the
vertices at time~$T$; all survivor positions and ghost
positions lie in~$\mathcal{Y}$. Each
source
vertex represents an initial particle; a \emph{path} from
$x_I$ to a target vertex~$y$ represents the particle's
trajectory, with weight $W(x_I \to y)$ equal to the path
generating function on the spacetime graph~$D$.

When two particles occupy the same vertex, they
\emph{annihilate}: both are destroyed and a \emph{twin pair}
of ghosts emerges from the collision point. The ghosts do not
remember which particles produced them; only the twin
pairing is preserved as the ghosts drift apart.
In particular, particles that share a starting position
annihilate instantly: the collision occurs at time zero
and only ghost paths emerge.

With $k$ pairwise annihilations---each producing one ghost
pair; a single collision vertex may host several
(\Cref{def:collision-diagram})---we have $s = n - 2k$
survivors and $2k$ ghosts (forming $k$ pairs), preserving
the total count of $n$ entities.

\subsection{Actors and roles}
\label{sec:setup-entities}

The theatrical names below anticipate the proof
(\Cref{sec:proof}), where a \emph{performance} records what
happened on stage and a \emph{casting} assigns actors to
roles.

\begin{definition}[Actors]\label{def:actors}
The \emph{actor set} $\Actors = \{1, \ldots, n\}$ indexes the initial
particles. Actor $I \in \Actors$ starts at position $x_I$.
\end{definition}

\begin{definition}[Roles]\label{def:roles}
The \emph{role set} (final entities) is:
\[
\Roles = \Survivors \cup \Ghosts,
\]
where:
\begin{itemize}
\item $\Survivors = \{1, \ldots, s\}$: survivor slots (with $s = n - 2k$);
\item $\Ghosts = \{1, \ldots, k\} \times \{1, 2\}$: ghost slots.
\end{itemize}
The cardinality is $s + 2k = n$, matching $|\Actors|$.
\end{definition}

\begin{definition}[Role order]\label{def:role-order}
We order the role set $\Roles$ linearly: the survivor slots
first, in their natural order $1 < \cdots < s$, followed by
the ghost slots in the lexicographic order
$(1,1) < (1,2) < (2,1) < \cdots < (k,2)$. Whenever we take
the sign $\sgn(\pi)$ of a bijection
$\pi\colon \Actors \to \Roles$, it is the sign of the
permutation obtained by reading~$\pi$ through this order on
the roles and the natural order on the actors.
\end{definition}

Each role $f \in \Roles$ has a final position $y_f$:
\begin{itemize}
\item Survivor slot $\ell \in \{1, \ldots, s\}$: final position $y_\ell$
  (with $y_1 \preceq \cdots \preceq y_s$, weakly increasing;
  see \Cref{rem:weak-survivors});
\item Ghost slot $(j, m) \in \{1, \ldots, k\} \times \{1, 2\}$:
  final position $y_{(j,m)}$.
\end{itemize}
We write $a_j = y_{(j,1)}$ and $b_j = y_{(j,2)}$ as shorthand for
the final positions of the two ghosts in pair~$j$.

\subsection{Collision diagrams and anchors}
\label{sec:setup-collision}

\begin{definition}[Collision diagram]\label{def:collision-diagram}
A \emph{collision diagram} is a graph embedded in the spacetime
graph $D$ with three types of vertices:
\begin{itemize}
\item \emph{Initial vertices} $x_1, \ldots, x_n$: each has
  in-degree~$0$ and out-degree~$1$, unless the vertex is
  also a collision vertex (see below).
\item \emph{Collision vertices}: each is reached by $m \geq 2$
  arriving particle paths (possibly of length zero, when
  initial particles share the vertex), has
  $d \in \{0, 1\}$ departing particle paths (possibly of
  length zero, when a surviving particle ends at the
  collision vertex), and even total count
  $m + d$ (so that all arriving particles
  can be paired). The collision produces $(m - d)/2$
  ghost pairs (and one continuing particle if $d = 1$);
\item \emph{Survivor vertices} $y_1, \ldots, y_s$: each has
  in-degree~$1$ and out-degree~$0$, unless the vertex is
  also a collision vertex (see below).
\end{itemize}
Each initial vertex has a directed path to either a collision
vertex or a survivor vertex. Distinct paths meet only at
collision vertices (no crossings between collisions).
The collision diagram records which collisions occur and where,
but not what happens after each collision; the ghost paths are
specified separately as part of the performance
(\Cref{def:performance}).
When two initial particles share a vertex, that vertex
serves simultaneously as initial vertex and collision
vertex; the arriving paths have length zero. Symmetrically,
when an annihilation occurs at a surviving particle's final
vertex, that vertex serves simultaneously as collision
vertex and survivor vertex; the departing path has length
zero.
\end{definition}

\begin{definition}[Anchors]\label{def:anchors}
The \emph{anchors} of a collision diagram are the pairs
$(v, i)$, where $v$ is a collision vertex producing
$(m-d)/2$ ghost pairs and $i \in \{1, \ldots, (m-d)/2\}$.
The ghost pairs born at~$v$ correspond bijectively to the
anchors of~$v$: each anchor carries exactly one ghost
pair, and a ghost pair is an anchor together with its
ordered pair of ghost paths (\Cref{def:performance}).
The anchors make the ghost pairs distinguishable: two
pairs born at the same vertex sit at different anchors,
even when their ghost paths coincide.
\end{definition}

\subsection{Performances}
\label{sec:setup-performances}

\begin{definition}[Performance]\label{def:performance}
An \emph{annihilation performance} $\perf$ specifies:
\begin{itemize}
\item A collision diagram: which particles annihilated and where;
\item Final positions for each survivor;
\item For each anchor (\Cref{def:anchors}): an ordered
  pair of ghost paths from its collision vertex---the
  ghost pair at that anchor;
\item A global numbering of the ghost pairs: a bijection
  from the anchors to $\{1, \ldots, k\}$. Ghost pair~$j$
  is the pair at the anchor numbered~$j$, with ghost paths
  $\Gamma_{j,1}, \Gamma_{j,2}$ to positions $a_j$
  and~$b_j$.
\end{itemize}
The \emph{weight} of a performance is
\[
w(\perf) =
  \prod_{\text{edges } e \text{ of diagram}} w(e)
  \cdot \prod_{j=1}^{k} w(\Gamma_{j,1})\, w(\Gamma_{j,2}),
\]
where $\Gamma_{j,1}$ is the ghost path to position
$a_j = y_{(j,1)}$ and $\Gamma_{j,2}$ is the ghost path to
position $b_j = y_{(j,2)}$.
\end{definition}

Performances are \emph{role-based}: they specify what
happened (collisions, final positions) without tracking
which initial particle ended where. 

The two bookkeeping
devices serve two purposes. The anchors make the ghost pairs
\emph{distinguishable} even if they are born at the same
vertex. The numbering gives the final state
\emph{coordinates}: the anchor set varies with the collision
diagram, and the numbering transports the anchor-indexed
data to the fixed index set~$\Roles$, so that the final
positions form a point $(y_f)_{f \in \Roles}$ of a product
space. This is what makes probability distributions of final
states tractable: they become measures on the product space,
as used for the continuous processes of
\Cref{sec:continuous}. Ghosts remain anonymous in the sense
of \Cref{sec:setup-annihilation}: neither the anchors nor
the numbering retain any memory of which particles created a
pair.

The numbering admits a dynamic description, implemented by
the companion code~\cite{Sniady2026companion}: sweep the
collisions in time order and insert each newborn pair at any
one of the positions of the current list of pairs ($i$
choices for the $i$-th pair); the final list order is the
numbering, and the $k!$ insertion histories of an evolution
realize its $k!$ numberings. In the probability regime the
insertion positions are uniform random choices, so that
$\det(M)/k!$---with $\det(M)$ as in
\Cref{thm:annihilation} below---is the probability of a
specific numbered final state.

\begin{example}[Two particles, one step]
\label{ex:performance-smallest} 
Two particles start at $x_1 = 0$ and $x_2 = 2$ and take
one $\pm 1$ step each, every edge of weight~$1$. There
are four evolutions of the two walkers. In one of them
both walkers step to the vertex~$1$ and annihilate: its
collision diagram consists of one collision vertex with
two arriving edges and $d = 0$; one ghost pair is born
there, and since the time horizon ends at the collision,
both ghost paths have length zero, $a_1 = b_1 = 1$. With
$k = 1$ the numbering and the insertion are trivial. The
weight of this performance is the product of the two
diagram edges, namely~$1$. The other three evolutions are
$k = 0$ performances---the collision diagram is just the
two particle paths---each of weight~$1$. In the count
regime these weights count performances; with step
weights~$\tfrac12$ each performance has weight
(probability)~$\tfrac14$. \Cref{ex:z-smallest} matches
these weights against the determinant.
\end{example}

\begin{remark}[Several ghost pairs at one vertex]
\label{rem:multipair-vertex}
When a collision vertex produces several ghost pairs, 
the anchors
matter exactly when two pairs from one vertex carry
identical ghost paths. For instance, for
$2k$ particles starting at a common vertex and observed at
time zero, all ghost paths have length zero and coincide,
and the $k!$ numberings of the $k$ anchors give $k!$
performances of weight~$1$, matching $\det(M) = k!$
(\Cref{rem:formula-edge-cases}).
\end{remark}

\Cref{fig:annihilation-performance} shows a larger example
with five particles and two collisions. The spacetime graph
there is the lattice $\ZZ^2$ with North and East steps: the
edges lead from $(u, v)$ to $(u, v+1)$ and to $(u+1, v)$,
so time flows toward the upper right. The source
set~$\mathcal{X}$ (the five starting vertices) and the
target set~$\mathcal{Y}$ (the final vertices) lie on two
staircase-shaped boundaries, each linearly ordered
by~$\prec$ from the northwest to the southeast; monotone
lattice paths between such boundaries satisfy the planarity
assumptions of \Cref{def:planar}, as for the examples of
\Cref{sec:setup-planarity}.

%
%

\begin{figure}[t]
\centering

\tikzset{
  style1/.style={colA, line width=2pt, solid,
    decorate, decoration={coil, segment length=3pt, amplitude=1pt}},
  style2/.style={colB, line width=1.5pt, solid,
    decorate, decoration={zigzag, segment length=4pt, amplitude=1pt}},
  style3/.style={colC, line width=1.5pt, solid,
    decorate, decoration={snake, segment length=5pt, amplitude=0.8pt}},
  style4/.style={colBC, line width=2pt, solid},
  style5/.style={colP, line width=2pt, solid, double, double distance=1pt},
  ghostPair1/.style={colGone, line width=1.5pt, dashed},
  ghostPair1double/.style={colGone, line width=1.5pt, dashed,
    double, double distance=2pt},
  ghostPair2/.style={black, line width=1pt, dotted},
  ghostPair2double/.style={black, line width=1pt, dotted,
    double, double distance=2pt},
  startnode/.style={circle, fill, inner sep=2.5pt},
  collisionnode/.style={circle, fill=black, inner sep=3pt},
  survivornode/.style={circle, fill=colP, inner sep=3pt},
  ghostend/.style={circle, draw, line width=1pt, fill=white, inner sep=2pt},
}

\begin{tikzpicture}[scale=0.20]
  \begin{scope}
    \clip (-1.5, 2.5) rectangle (34.5, 34.5);
    \draw[gray!30, thin, step=2] (-2, 2) grid (36, 36);
  \end{scope}

  \draw[ghostPair1] (8,18) -- (8,19) -- (8,20) -- (8,21) -- (8,22) -- (8,23) -- (8,24);
  \draw[ghostPair1, transform canvas={shift={(0,0.15)}}]
                    (8,24) -- (9,24) -- (10,24) -- (11,24) -- (12,24) -- (13,24) -- (14,24)
                    -- (15,24) -- (16,24) -- (17,24) -- (18,24) -- (19,24) -- (20,24)
                    -- (21,24) -- (22,24) -- (23,24) -- (24,24);
  \draw[ghostPair1double] (8,18) -- (9,18) -- (10,18) -- (11,18) -- (12,18) -- (13,18)
                          -- (14,18) -- (15,18) -- (16,18) -- (17,18) -- (18,18) -- (19,18)
                          -- (20,18) -- (21,18) -- (22,18) -- (23,18) -- (24,18) -- (25,18)
                          -- (26,18) -- (27,18) -- (28,18) -- (28,19) -- (28,20);

  \draw[ghostPair2] (14,24) -- (15,24) -- (16,24) -- (16,25) -- (16,26) -- (16,27)
                    -- (16,28) -- (16,29) -- (16,30) -- (16,31) -- (16,32);
  \draw[ghostPair2double] (14,24) -- (14,25) -- (14,26) -- (14,27) -- (14,28)
                          -- (15,28) -- (16,28) -- (17,28) -- (18,28) -- (19,28) -- (20,28);

  \draw[style1] (0,20) -- (0,21) -- (0,22) -- (0,23) -- (0,24)
                -- (1,24) -- (2,24) -- (3,24) -- (4,24) -- (5,24) -- (6,24) -- (7,24)
                -- (8,24) -- (9,24) -- (10,24) -- (11,24) -- (12,24) -- (13,24) -- (14,24);
  \draw[style2] (2,16) -- (2,17) -- (2,18)
                -- (3,18) -- (4,18) -- (5,18) -- (6,18) -- (7,18) -- (8,18);
  \draw[style3] (4,12) -- (5,12) -- (6,12) -- (7,12) -- (8,12)
                -- (8,13) -- (8,14) -- (8,15) -- (8,16) -- (8,17) -- (8,18);
  \draw[style4] (6,8) -- (7,8) -- (8,8) -- (9,8) -- (10,8) -- (11,8) -- (12,8) -- (13,8) -- (14,8)
                -- (14,9) -- (14,10) -- (14,11) -- (14,12) -- (14,13) -- (14,14) -- (14,15)
                -- (14,16) -- (14,17) -- (14,18) -- (14,19) -- (14,20) -- (14,21) -- (14,22)
                -- (14,23) -- (14,24);
  \draw[style5] (8,4) -- (9,4) -- (10,4) -- (11,4) -- (12,4) -- (13,4) -- (14,4) -- (15,4)
                -- (16,4) -- (17,4) -- (18,4) -- (19,4) -- (20,4) -- (21,4) -- (22,4) -- (23,4)
                -- (24,4) -- (25,4) -- (26,4) -- (27,4) -- (28,4) -- (29,4) -- (30,4) -- (31,4) -- (32,4)
                -- (32,5) -- (32,6) -- (32,7) -- (32,8) -- (32,9) -- (32,10) -- (32,11) -- (32,12);

  \node[left, fill=white] at (-0.2, 20) {$x_1$};
  \node[left, fill=white] at (1.8, 16) {$x_2$};
  \node[left, fill=white] at (3.8, 12) {$x_3$};
  \node[left, fill=white] at (5.8, 8) {$x_4$};
  \node[left, fill=white] at (7.8, 4) {$x_5$};

  \node[below left=0.1, fill=white] at (8, 18) {$c_1$};
  \node[below left=0.1, fill=white] at (14, 24) {$c_2$};

  \node[right=0.2, fill=white, inner sep=1pt] at (32, 12) {$s_1$};

  \node[right=0.2, fill=white, inner sep=1pt] at (24, 24) {$a_1$};
  \node[right=0.2, fill=white, inner sep=1pt] at (28, 20) {$b_1$};
  \node[above=0.2, fill=white, inner sep=1pt] at (16, 32) {$a_2$};
  \node[right=0.2, fill=white, inner sep=1pt] at (20, 28) {$b_2$};

  \node[startnode, colA] at (0, 20) {};
  \node[startnode, colB] at (2, 16) {};
  \node[startnode, colC] at (4, 12) {};
  \node[startnode, colBC] at (6, 8) {};
  \node[startnode, colP] at (8, 4) {};

  \node[collisionnode] at (8, 18) {};
  \node[collisionnode] at (14, 24) {};

  \node[survivornode] at (32, 12) {};
  \node[ghostend, draw=colGone] at (24, 24) {};
  \node[ghostend, draw=colGone] at (28, 20) {};
  \node[ghostend, draw=black] at (16, 32) {};
  \node[ghostend, draw=black] at (20, 28) {};

  \begin{scope}[shift={(36, 10)}]
    \node[anchor=west, font=\small\bfseries] at (0, 22) {Particles};
    \draw[style1] (0, 20) -- (2, 20); \node[anchor=west] at (2.5, 20) {\small $1$};
    \draw[style2] (0, 18) -- (2, 18); \node[anchor=west] at (2.5, 18) {\small $2$};
    \draw[style3] (0, 16) -- (2, 16); \node[anchor=west] at (2.5, 16) {\small $3$};
    \draw[style4] (0, 14) -- (2, 14); \node[anchor=west] at (2.5, 14) {\small $4$};
    \draw[style5] (0, 12) -- (2, 12); \node[anchor=west] at (2.5, 12) {\small $5$ (survivor)};

    \node[anchor=west, font=\small\bfseries] at (0, 9) {Ghost pairs};
    \draw[ghostPair1] (0, 7) -- (2, 7);
    \draw[ghostPair1double] (0, 5) -- (2, 5);
    \node[anchor=west] at (2.5, 6) {\small pair 1};
    \draw[ghostPair2] (0, 2) -- (2, 2);
    \draw[ghostPair2double] (0, 0) -- (2, 0);
    \node[anchor=west] at (2.5, 1) {\small pair 2};
  \end{scope}

\end{tikzpicture}

\caption{An annihilation performance on the lattice $\ZZ^2$ with North/East
steps; time flows toward the upper right, and both the sources and
the targets are ordered by~$\prec$ from the northwest to the
southeast. Five particles start at $x_1, \ldots, x_5$. Particles~$2$ and~$3$
meet at $c_1$; both are destroyed and ghost pair~$1$ emerges (dashed paths).
Particles~$1$ and~$4$ meet at $c_2$; both are destroyed and ghost pair~$2$
emerges (dotted paths). Particle~$5$ survives, reaching~$s_1$. Within each
ghost pair, one ghost follows a single line, the other a double line.
Ghost pairs are distinguished by line pattern (dashed vs.\
dotted)---they carry no memory of which particles were
destroyed. In this example both pairs happen to end in the
order $a_j \preceq b_j$, so both ghost signs
(\Cref{def:ghost-sign}) are $\varepsilon_j = +1$; the
reversed order $b_j \prec a_j$ is equally possible and is
recorded by $\varepsilon_j = -1$.}
\label{fig:annihilation-performance}

\end{figure}

\subsection{Final state}
\label{sec:setup-final-state}

\begin{definition}[Ghost sign]\label{def:ghost-sign}
For ghost pair $j$, the \emph{ghost sign} is:
\[
\varepsilon_j = \begin{cases}
+1 & \text{if } a_j \preceq b_j, \\
-1 & \text{if } b_j \prec a_j.
\end{cases}
\]
When $a_j = b_j$ (ghosts at the same final position), we have
$\varepsilon_j = +1$ by convention; the choice is immaterial
(\Cref{rem:formula-edge-cases}). For continuous state spaces,
this case has probability zero.
\end{definition}

\begin{definition}[Final state]\label{def:final-state}
The \emph{final state} $\FinalState$ specifies:
\begin{itemize}
\item The number $k$ of ghost pairs, equivalently of
  pairwise annihilations (with $2k \leq n$);
\item Final positions $y_f$ for each $f \in \Roles$.
\end{itemize}
The \emph{sign} of the final state is the product of ghost signs:
$\sgn \FinalState = \prod_{j=1}^{k} \varepsilon_j$.
\end{definition}

\section{The annihilation formula}\label{sec:formula}

The matrix~$M$ defined below has one row per actor and one
column per role. Its Leibniz expansion sums over all
bijections $\pi\colon \Actors \to \Roles$, but only those
consistent with the ghost configuration should contribute.
We introduce formal variables that track the relative
ordering of the two actors assigned to each ghost pair;
their products vanish on the inconsistent terms, so the
determinant itself performs the selection.

\subsection{Formal variables}
\label{sec:formula-variables}

For each ghost pair $j \in \{1, \ldots, k\}$, introduce
formal variables $\tau^{(j)+}_I$, $\tau^{(j)-}_I$ for each
particle $I$. Their products evaluate to scalars, by a rule
that depends on the ghost sign~$\varepsilon_j$
(\Cref{def:ghost-sign}); for $I \neq J$:
\begin{equation}\label{eq:formal}
\tau^{(j)+}_I \, \tau^{(j)-}_J = \begin{cases}
-1 & \text{if } \varepsilon_j = +1 \text{ and } I > J, \\
+1 & \text{if } \varepsilon_j = -1 \text{ and } I < J, \\
\phantom{+}0 & \text{otherwise}.
\end{cases}
\end{equation}
(The case $I = J$ does not arise, since bijections assign
distinct actors to distinct slots.)

Concretely, the product rule~\eqref{eq:formal} keeps a term
exactly when the two actors assigned to ghost pair~$j$ are
ordered consistently with the ghost sign---the higher-indexed
actor in slot~$(j,1)$ when $\varepsilon_j = +1$, the
lower-indexed one when $\varepsilon_j = -1$---and the
surviving term carries the sign $-\varepsilon_j$.

All formal variables commute with each other and with the
path weights. The relations~\eqref{eq:formal} suffice to
evaluate the determinant: every monomial of the Leibniz
expansion contains, for each ghost pair~$j$, exactly one
factor $\tau^{(j)+}_I$ (from column $(j,1)$) and one factor
$\tau^{(j)-}_J$ (from column $(j,2)$) with $I \neq J$, so
each monomial evaluates uniquely to a signed product of path
weights, possibly zero.

\subsection{The matrix}
\label{sec:formula-matrix}

Define the $n \times n$ matrix $M$ with:
\begin{itemize}
\item Rows indexed by particles $I \in \Actors$;
\item Columns indexed by roles $f \in \Roles$, in the role
  order (\Cref{def:role-order}): survivor slots
  $1, \ldots, s$ followed by ghost slots
  $(1,1), (1,2), \ldots, (k,1), (k,2)$.
\end{itemize}
Entries:
\[
M_{I,f} = \begin{cases}
\phantom{\tau^{(j)+}_I \,} W(x_I \to y_\ell) & \text{survivor column } \ell, \\[0.3em]
\tau^{(j)+}_I \, W(x_I \to a_j) & \text{column for ghost slot } (j,1), \\[0.3em]
\tau^{(j)-}_I \, W(x_I \to b_j) & \text{column for ghost slot } (j,2).
\end{cases}
\]

\subsection{The theorem}
\label{sec:formula-theorem}

\begin{theorem}[Annihilation formula with ghosts]\label{thm:annihilation}
For a fixed final state $\FinalState$, the total weight of
performances is:
\[
Z = \det(M).
\]
When $k = 0$ (no collisions), there are no formal variables
and the formula is the Karlin--McGregor determinant.
\end{theorem}

\noindent
The proof occupies \Cref{sec:proof}.

\begin{remark}[Edge cases of the boundary conditions]
\label{rem:formula-edge-cases}
\label{rem:weak-survivors}
Two boundary configurations refine the reading of
\Cref{thm:annihilation}.

\emph{Weakly increasing survivor positions:} the ordering
assumption on the survivor positions (\Cref{def:roles}) is
weak, so the theorem covers two survivors sharing a final
position. Such particles would occupy the same vertex and
annihilate rather than survive, so no performance attains
this final state and $Z = 0$; on the matrix side, the two
survivor columns coincide, so $\det(M) = 0$ as well.
(Survivor positions may freely coincide with ghost
positions: ghosts share vertices with any path.)

\emph{Coincident ghosts:} when the two ghosts of pair~$j$
share their final position ($a_j = b_j$),
\Cref{def:ghost-sign} sets $\varepsilon_j = +1$ by
convention, and $\det(M)$ does not depend on the choice.
Indeed, the two ghost columns of pair~$j$ then carry equal
weights $W(x_I \to a_j)$, and switching~$\varepsilon_j$
exchanges which of the two slot assignments of an actor pair
survives the evaluation rule~\eqref{eq:formal}---changing
the permutation sign by a transposition---while
simultaneously flipping the formal
factor~$-\varepsilon_j$; the two sign changes cancel.
The two mechanisms cooperate in the fully degenerate case
of $n$ particles starting at a common vertex and observed
at time zero: the formula assigns weight~$k!$ (one unit
per numbering; see \Cref{rem:multipair-vertex}) to the
forced outcome---the maximal number of annihilations,
every final position at the starting vertex---and
weight~$0$ to every other final state, whose identical
survivor columns make the determinant vanish.
\end{remark}

\subsection{Examples}
\label{sec:formula-examples}

\begin{example}[What $Z$ counts]
\label{ex:z-smallest}
Continue \Cref{ex:performance-smallest}: two particles at
$0$ and~$2$, one $\pm 1$ step, all edge weights~$1$. For
the three final states with two survivors, at $(-1, 1)$,
$(-1, 3)$, and $(1, 3)$, the matrix~$M$ is the
Karlin--McGregor matrix, and
\[
Z = \det\begin{pmatrix} 1 & 1 \\ 0 & 1 \end{pmatrix}
  = \det\begin{pmatrix} 1 & 0 \\ 0 & 1 \end{pmatrix}
  = \det\begin{pmatrix} 1 & 0 \\ 1 & 1 \end{pmatrix}
  = 1,
\]
one performance each. The state $(1, 1)$ for $k=0$---both
walkers surviving at the same vertex---has identical
survivor columns and $Z = 0$
(\Cref{rem:formula-edge-cases}): the walkers meeting
at~$1$ annihilate instead. Their annihilation is the
$k = 1$ final state with ghost pair positions
$a_1 = b_1 = 1$ and
\begin{align*}
M &= \begin{pmatrix}
\tau^{(1)+}_1 & \tau^{(1)-}_1 \\
\tau^{(1)+}_2 & \tau^{(1)-}_2
\end{pmatrix},
\\
Z &= \det(M)
  = \tau^{(1)+}_1 \tau^{(1)-}_2
  - \tau^{(1)+}_2 \tau^{(1)-}_1
  = 0 - (-1) = 1:
\end{align*}
the tie gives $\varepsilon_1 = +1$, so by the evaluation
rule~\eqref{eq:formal} the first product ($I = 1 < J = 2$)
vanishes, while the second ($I = 2 > J = 1$) evaluates
to~$-1$---the surviving term puts actor~$2$ in
slot~$(1,1)$. The four values~$1$ account for the four
evolutions; with step weights~$\tfrac12$ the total mass
is $4 \cdot \tfrac14 = 1$.
\end{example}

\begin{example}[Multi-annihilation: the funnel]
\label{ex:funnel}
Spacetime graphs are not restricted to the lattice.
Consider the \emph{funnel} (\Cref{fig:funnel}): four
sources $x_1 \prec x_2 \prec x_3 \prec x_4$, each with a
single outgoing edge of weight~$1$, all four leading to a
common vertex~$v$; two edges leave~$v$, toward targets
$y_- \prec y_+$, with weights $\alpha$ and~$\beta$. Planarity
holds trivially: every pair of paths shares the
vertex~$v$. All four particles
annihilate at~$v$: the collision vertex has $m = 4$
arriving paths and $d = 0$, so two ghost pairs are born
together, attached to the two anchors of~$v$
(\Cref{def:anchors}), with
$s = 0$ and $k = 2$; each of the four ghosts then takes
one of the two outgoing edges.

\begin{figure}[t]
\centering
\begin{tikzpicture}[scale=1.0,
    vertex/.style={circle, fill, inner sep=1.5pt},
    collision/.style={circle, fill=black, inner sep=2pt},
    ghostend/.style={circle, draw, line width=1pt,
      fill=white, inner sep=2pt},
    ghostPair1/.style={colGone, line width=1.5pt, dashed},
    ghostPair1double/.style={colGone, line width=1.5pt, dashed,
      double, double distance=2pt},
    ghostPair2/.style={black, line width=1pt, dotted},
    ghostPair2double/.style={black, line width=1pt, dotted,
      double, double distance=2pt}]
  \node[vertex, colA] (x1) at (0, 0) {};
  \node[below] at (x1) {$x_1$};
  \node[vertex, colB] (x2) at (2, 0) {};
  \node[below] at (x2) {$x_2$};
  \node[vertex, colC] (x3) at (4, 0) {};
  \node[below] at (x3) {$x_3$};
  \node[vertex, colD] (x4) at (6, 0) {};
  \node[below] at (x4) {$x_4$};
  \node[collision] (v) at (3, 1.5) {};
  \node[below=4pt] at (v) {$v$};
  \node[ghostend] (a) at (2, 3) {};
  \node[above] at (a) {$y_-$};
  \node[ghostend] (b) at (4, 3) {};
  \node[above] at (b) {$y_+$};
  \draw[particleA] (x1) -- (v);
  \draw[particleB] (x2) -- (v);
  \draw[particleC] (x3) -- (v);
  \draw[particleD] (x4) -- (v);
  \draw[ghostPair1] (v) to[bend left=18] (a);
  \draw[ghostPair1double] (v) to[bend left=18] (b);
  \draw[ghostPair2] (v) to[bend right=18] (a);
  \draw[ghostPair2double] (v) to[bend right=18] (b);
  \node[font=\scriptsize] at (2.05, 2.15) {$\alpha$};
  \node[font=\scriptsize] at (3.95, 2.15) {$\beta$};
  \begin{scope}[shift={(7.0, 0.7)}]
    \draw[ghostPair1] (0, 1.8) -- (0.9, 1.8);
    \draw[ghostPair1double] (0, 1.4) -- (0.9, 1.4);
    \node[anchor=west, font=\scriptsize] at (1.0, 1.6)
      {pair at anchor $(v, 1)$};
    \draw[ghostPair2] (0, 0.6) -- (0.9, 0.6);
    \draw[ghostPair2double] (0, 0.2) -- (0.9, 0.2);
    \node[anchor=west, font=\scriptsize] at (1.0, 0.4)
      {pair at anchor $(v, 2)$};
  \end{scope}
\end{tikzpicture}
\caption{The funnel of \Cref{ex:funnel}, drawn in the
visual style of the lattice figures; time runs upward.
Each source has a single outgoing edge, and all four lead
to~$v$, so every pair of paths shares~$v$ and planarity
holds trivially. All
four particles annihilate at~$v$ ($m = 4$, $d = 0$), and two
ghost pairs are born: the pair at anchor $(v, 1)$ is drawn
dashed, the pair at anchor $(v, 2)$ dotted. Within each pair,
the ghost path to~$y_-$ (the first slot) is a single line and
the ghost path to~$y_+$ (the second slot) a double line, as in
\Cref{fig:annihilation-performance}. The final state shown is
the identical-pair state: both pairs reach the positions
$(y_-, y_+)$, and the two performances above it differ only by
which anchor carries the number~$1$.}
\label{fig:funnel}
\end{figure}

A final state assigns positions
$(a_1, b_1), (a_2, b_2) \in \{y_-, y_+\}^2$ to the two
numbered pairs---sixteen final states in all. Every one
of them has
\[
Z = 2!\,\alpha^{\#_-}\,\beta^{\#_+},
\]
where $\#_-$ and $\#_+$ count the slots at~$y_-$ and
at~$y_+$. The factor~$2!$ counts the two numberings of the
two anchors of~$v$, so each physical outcome is reached by
two performances---also when the two pairs follow
identical ghost paths, since the numberings differ on the
anchors (\Cref{rem:multipair-vertex}). On the
determinant side, for each of the sixteen final states
exactly six terms of the Leibniz expansion survive the
evaluation rule~\eqref{eq:formal}---the three ways to
pair the four actors, times the two numberings of the
pairs---and each carries the same weight
$\alpha^{\#_-}\beta^{\#_+}$. Grouped by the pairing, the
signs contribute $+2$, $-2$, $+2$, giving
$\det(M) = 2\,\alpha^{\#_-}\beta^{\#_+}$; this grouping
by pairings, run uniformly over all states, is exactly
how the Pfaffian formula of \Cref{sec:pfaffian} will
emerge.

In the count regime ($\alpha = \beta = 1$) every final
state has $Z = 2$, and the sixteen final states carry
$32 = 2! \cdot 16$ performances in total. In the
probability regime ($\alpha = p$,
$\beta = q = 1 - p$), each numbered final state has
probability $Z/2!$, and the sixteen values sum to
$(p + q)^4 = 1$. 

For an odd number of particles the funnel additionally
produces a survivor standing at the collision vertex
($d = 1$, with a departing path of length zero;
\Cref{def:collision-diagram}). The companion code
verifies both funnel instances
exactly~\cite{Sniady2026companion}.
\end{example}

\section{Proof of the annihilation formula}\label{sec:proof}

This proof parallels the coalescence proof
in~\cite{SU2026coalescence}: identical four-part structure
(castings, attribution, rehearsal, sign-reversing involution)
and identical involution mechanism (segment swap at the first
wrong crossing). The annihilation proof is simpler because
annihilation destroys both particles at a collision---no heir
continues---reducing the case analysis in the involution.
Nevertheless, ghost pairs differ enough from single ghosts that
the success criterion, attribution map, and involution case
analysis all require reworking. The present proof is
self-contained; familiarity with the companion paper is not
required.

\subsection{Performances vs.\ castings}
\label{sec:proof-strategy}

Throughout the proof, we fix a final state~$\FinalState$
(\Cref{sec:setup-final-state}): the number~$k$ of collisions,
survivor positions $y_1, \ldots, y_s$, and ghost pair positions
$(a_j, b_j)$ for $j = 1, \ldots, k$. The generating function
$Z = Z_{\FinalState}$ counts \emph{performances}
for~$\FinalState$: complete descriptions of what collisions occur
and where the final entities end up.

\subsubsection{Performances}

A performance is \textbf{role-based}---it specifies outcomes
without tracking which particle ``becomes'' which final entity.
In probabilistic terms, each performance is an elementary event
in the sample space.

Consider \Cref{fig:intro-annihilation}: four particles start at
$x_1 < x_2 < x_3 < x_4$, particles~$2$ and~$3$ collide at~$c$,
and the system ends with two survivors at $y_1 < y_2$ and a ghost
pair at $(a, b)$. The role-based description records:
\begin{itemize}
\item A collision occurred at spacetime point~$c$.
\item The ghost pair emerged from that collision, reaching
  positions $a$ and~$b$.
\item Two particles survived, reaching $y_1$ and $y_2$.
\item The paths taken by each entity.
\end{itemize}
This description says nothing about \emph{which} of particles~$2$
and~$3$ ended at which ghost position---only that the collision
produced a ghost pair. In theater terms: the script specifies
where each character should stand at the final scene, but does
not specify which actor plays which character.

\subsubsection{The determinant perspective}

The Leibniz expansion of 
the determinant $\det(M)$ gives something different: a signed sum
over \emph{castings}---assignments of particles to final positions
via non-interacting paths. A casting is \textbf{actor-based}: it
tracks each particle's complete trajectory.
\Cref{fig:successful-casting} shows a successful casting where
line styles persist through the collision, revealing who went
where.

The proof establishes a bijection between performances and a
subset of castings (the ``successful'' ones), while the remaining
``failed'' castings cancel in pairs (\Cref{fig:failed-casting}).

\subfile{../figures/fig-castings}

\subsubsection{Outline of the argument}

The argument has four parts:
\begin{enumerate}
\item \textbf{Castings from the determinant.} The Leibniz expansion
  produces castings; the evaluation of the formal variables
  annihilates all but the \emph{candidate}
  castings, those consistent with the ghost configuration.
\item \textbf{Attribution.} Every performance determines a casting
  by tracking which particle ends where. Attribution always
  produces a candidate casting.
\item \textbf{Rehearsal.} We attempt to interpret each casting as a
  performance by scanning crossings in temporal order. This may
  succeed (``successful casting'') or fail (``failed casting'').
\item \textbf{The sign-reversing involution.} Failed castings pair
  up via segment swap and cancel. Only successful castings remain.
\end{enumerate}
The planarity assumptions enter only through two standalone
propositions (\Cref{sec:proof-planarity}).

\subsection{Castings}
\label{sec:proof-castings}

The Leibniz formula expands the determinant as a sum over bijections
$\pi\colon \Actors \to \Roles$:
\[
\det(M) = \sum_{\pi} \sgn(\pi) \prod_{I \in \Actors} M_{I, \pi(I)},
\]
where $\sgn(\pi)$ is taken with respect to the role order on
the columns (\Cref{def:role-order}).
Each matrix entry $M_{I,f}$ is either a path generating function
$W(x_I \to y_f)$ (for survivor columns) or
$\tau^{(j)\pm}_I W(x_I \to y_{(j,m)})$ (for ghost columns).

\begin{definition}[Casting]\label{def:casting}
A \emph{casting} $(\pi, \paths)$ consists of:
\begin{itemize}
\item A bijection $\pi\colon \Actors \to \Roles$;
\item A path family $\paths = \{P_I\}_{I \in \Actors}$ where $P_I$ goes
  from $x_I$ to $y_{\pi(I)}$.
\end{itemize}
The \emph{weight} of a casting is
\[w(\casting) = \prod_{I \in \Actors} w(P_I).\]
\end{definition}

Crucially, castings are pure geometry: the paths are non-interacting.
They may cross freely---a crossing is simply a shared vertex with no
physical consequence.

\begin{definition}[Candidate bijection]\label{def:candidate}
A bijection $\pi\colon \Actors \to \Roles$ is a \emph{candidate}
for $\FinalState$ if, for each ghost pair $j \in \{1, \ldots, k\}$:
\begin{itemize}
\item when $\varepsilon_j = +1$:
  $\pi^{-1}(j,1) > \pi^{-1}(j,2)$;
\item when $\varepsilon_j = -1$:
  $\pi^{-1}(j,1) < \pi^{-1}(j,2)$.
\end{itemize}
In both cases, the higher-indexed particle maps to the
leftward ghost position.
\end{definition}

Candidacy arises from the evaluation of the formal variables.
Each ghost
pair~$j$ contributes two columns to $\det(M)$: column $(j,1)$
with entries $\tau^{(j)+}_I$ and column $(j,2)$ with entries
$\tau^{(j)-}_I$. When a bijection~$\pi$ sends particles $I$ to
$(j,1)$ and $J$ to $(j,2)$, the product
rule~\eqref{eq:formal} evaluates
$\tau^{(j)+}_I \tau^{(j)-}_J$ to $-\varepsilon_j$ if the pair
$(I, J)$ is ordered consistently with the ghost
sign~$\varepsilon_j$, and to zero otherwise: exactly the
bijections satisfying the candidacy condition survive.

Write $\Candidates$ for the set of candidate bijections;
a casting is a \emph{candidate casting} if its bijection
is a candidate. After
evaluating the formal variables, only candidates survive. Each
candidate
$\pi$ contributes a sign from the formal variable product
(determined by the ordering of actors within each ghost pair)
in addition to $\sgn(\pi)$:
\[
\det(M)
= \sum_{\pi \in \Candidates} (\text{formal sign}) \cdot \sgn(\pi)
  \prod_{I} W(x_I \to y_{\pi(I)}).
\]
The sign analysis in \Cref{sec:proof-completion} shows that these
signs combine to give $+1$ for each successful casting.

\begin{remark}[No tournament needed]\label{rem:no-tournament}
In the companion paper the physical sign of a casting cannot
be produced by any linear order of the roles; it is defined
there through a tournament---a total but non-transitive
relation~\cite{SU2026coalescence}. The annihilation formula
avoids this: the ghost signs are absorbed by the formal
variables (\Cref{sec:formula-variables}), and the plain
permutation sign with respect to the role order suffices.
\end{remark}

\subsection{Two consequences of planarity}
\label{sec:proof-planarity}

The planarity assumptions (\Cref{def:planar}) enter the proof
through exactly two statements about the paths of a casting,
recorded here once; everything downstream---the stall analysis,
the involution, the sign identity---cites these two
propositions and never invokes the planarity assumptions
directly. This separation of the geometry from the
combinatorics follows the companion
paper~\cite{SU2026coalescence}. Recall that a \emph{crossing} of
two paths is simply a shared vertex
(\Cref{sec:proof-castings}).

\begin{proposition}[Order preservation]
\label{prop:order-preservation}
Let $A \subseteq \Actors$ be a set of actors whose casting
paths are pairwise non-crossing. Then the endpoints are in the
strict spatial order of the sources: $I < J$ in~$A$ implies
$y_{\pi(I)} \prec y_{\pi(J)}$.
\end{proposition}

\begin{proof}
Fix $I < J$ in~$A$, so that $x_I \preceq x_J$. Since $P_I$
and~$P_J$ share no vertex, their sources are distinct
($x_I \prec x_J$), and so are their endpoints. If the endpoints
were in reversed order, $y_{\pi(J)} \prec y_{\pi(I)}$, the
crossing property~\ref{item:crossing} would force $P_I$
and~$P_J$ to intersect. Hence $y_{\pi(I)} \prec y_{\pi(J)}$.
\end{proof}

\begin{proposition}[Crossing blocks are contiguous]
\label{prop:adjacent-crossings}
Let $A \subseteq \Actors$ be a set of actors, 
and let~$v$ be the first crossing among the paths
of~$A$ in the linear order (\Cref{def:spacetime-graph}). Then
the actors of~$A$ whose paths pass through~$v$ form a
contiguous block of~$A$: if $I < K < J$ in~$A$ and the paths
of $I$ and~$J$ pass through~$v$, so does the path of~$K$.
\end{proposition}

\begin{proof}
The sources satisfy $x_I \preceq x_K \preceq x_J$. Suppose
first $x_K = x_I$. Then $P_K$ and $P_I$ share their source
vertex~$u$---a crossing among the paths of~$A$. Since $u$ is
the source of the directed path~$P_I$ and $v$ lies on that
path, $u \lessdot v$ or $u = v$; the linear order extends the
time order, and no crossing precedes the first one~$v$, so
$u = v$ and $P_K$ passes through~$v$. The case $x_K = x_J$ is
symmetric.

Otherwise $x_I \prec x_K \prec x_J$, and the consecutive
collision property~\ref{item:consecutive}, applied to the
meeting of $P_I$ and~$P_J$ at~$v$, forces $P_K$ either to pass
through~$v$ or to intersect $P_I$ or~$P_J$ before~$v$. The
second alternative is a crossing among the paths of~$A$ that
precedes~$v$ in the linear order---impossible, since $v$ is the
first one. Hence $P_K$ passes through~$v$.
\end{proof}

Applied with $A$ the set of active particles,
\Cref{prop:adjacent-crossings} says that the particles
that rehearsal (\Cref{sec:proof-rehearsal}) finds at its next
crossing vertex form a contiguous block of the active set; in
particular, when exactly two active paths cross there, the two
particles are adjacent among the active ones.

\subsection{Attribution}
\label{sec:proof-attribution}

A performance specifies collision locations and final positions
but not identities. \emph{Attribution} constructs the underlying
casting by tracking which particle ends where. We first describe
what happens at a single collision, then explain how to combine
these local operations.

\subsubsection{The two-particle case}

At a pairwise annihilation, two path segments arrive and two
ghost paths depart. The \emph{far-side principle} determines the
gluing:

\begin{principle}[Far-side principle]\label{principle:swap}
At a collision between particles $I < J$ filling ghost
pair~$j$, each particle is glued to the ghost path departing
on the far side: the higher-indexed particle~$J$ to the
ghost at the leftward position, the lower-indexed
particle~$I$ to the ghost at the rightward position.
\end{principle}

\Cref{fig:simplified-annihilation} illustrates the case
$\varepsilon_j = +1$ (that is, $a_j \preceq b_j$).
Two ghost paths depart the collision vertex toward
positions $a_j$ and $b_j$, see
\cref{fig:simplified-annihilation-a}. The far-side principle
glues each incoming path to an outgoing ghost path,
see \cref{fig:simplified-annihilation-b}: particle~$J$
(higher index) is glued to the leftward ghost
at~$a_j$, and particle~$I$ (lower index) is glued
to the rightward ghost at~$b_j$. 

The case
$\varepsilon_j = -1$ (meaning $b_j \prec a_j$) is
symmetric (\Cref{fig:simplified-annihilation-negative}):
the labels $a_j$ and $b_j$ swap sides, so
particle~$J$ again goes to the leftward position---now
$b_j$---and particle~$I$ to the rightward one---now $a_j$.
In both cases, the
index ordering is reversed relative to the ghost pair's
spatial ordering.

%
%

\begin{figure}[t]
\centering

\subfloat[]{
\begin{tikzpicture}[scale=0.8]

\def\width{3.7}
\def\height{3}

\draw[gray, thick] (-\width/2, 0) -- (\width/2, 0);
\draw[gray, thick] (-\width/2, \height) -- (\width/2, \height);
\node[right] at (\width/2, 0) {$\mathcal{X}$};
\node[right] at (\width/2, \height) {$\mathcal{Y}$};

\coordinate (startI) at (-1.0, 0);
\coordinate (startJ) at (1.0, 0);

\coordinate (collision) at (0, 1.5);

\coordinate (endA) at (-1.0, \height);    
\coordinate (endB) at (1.0, \height);     

\draw[colA, line width=2.5pt] (startI) -- (collision);
\draw[colB, line width=1.5pt, decorate,
      decoration={snake, segment length=5pt, amplitude=1pt}]
    (startJ) -- (collision);

\draw[colGone, line width=1.5pt, dashed,
      dash pattern={on 4pt off 2pt}]
    (collision) -- (endA);
\draw[colGone, line width=1.5pt, dashed,
      dash pattern={on 4pt off 2pt},
      double, double distance=2pt]
    (collision) -- (endB);

\node[circle, fill=black, inner sep=2.5pt] at (collision) {};

\node[circle, fill=colA, inner sep=2.5pt] at (startI) {};
\node[circle, fill=colB, inner sep=2.5pt] at (startJ) {};

\node[circle, draw=colGone, fill=white, line width=1pt,
      inner sep=2pt] at (endA) {};
\node[circle, draw=colGone, fill=white, line width=1pt,
      inner sep=2pt] at (endB) {};

\node[below left] at (startI) {$I$};
\node[below right] at (startJ) {$J$};
\node[above left] at (endA) {$a_j$};
\node[above right] at (endB) {$b_j$};

\end{tikzpicture}
\label{fig:simplified-annihilation-a}
}
\hfill
\subfloat[]{
\begin{tikzpicture}[scale=0.8]

\def\width{3.7}
\def\height{3}

\draw[gray, thick] (-\width/2, 0) -- (\width/2, 0);
\draw[gray, thick] (-\width/2, \height) -- (\width/2, \height);
\node[right] at (\width/2, 0) {$\mathcal{X}$};
\node[right] at (\width/2, \height) {$\mathcal{Y}$};

\coordinate (startI) at (-1.0, 0);
\coordinate (startJ) at (1.0, 0);
\coordinate (V) at (0, 1.5);
\coordinate (endA) at (-1.0, \height);    
\coordinate (endB) at (1.0, \height);     

\draw[colA, line width=2.5pt] (startI) -- (V);
\draw[colA, line width=2.5pt, dashed,
      dash pattern={on 4pt off 2pt},
      double, double distance=2pt]
    (V) -- (endB);

\draw[colB, line width=1.5pt, decorate,
      decoration={snake, segment length=5pt, amplitude=1pt}]
    (startJ) -- (V);
\draw[colB, line width=1.5pt, dashed,
      dash pattern={on 5pt off 3pt}, decorate,
      decoration={snake, segment length=5pt, amplitude=1pt}]
    (V) -- (endA);

\node[circle, fill=white, inner sep=0pt, minimum size=6pt,
      line width=0.8pt, draw=black] at (V) {};

\node[circle, fill=colA, inner sep=2.5pt] at (startI) {};
\node[circle, fill=colB, inner sep=2.5pt] at (startJ) {};

\node[circle, draw=colB, fill=white, line width=1pt,
      inner sep=2pt] at (endA) {};
\node[circle, draw=colA, fill=white, line width=1pt,
      inner sep=2pt] at (endB) {};

\node[below left] at (startI) {$I$};
\node[below right] at (startJ) {$J$};
\node[above left] at (endA) {$a_j$};
\node[above right] at (endB) {$b_j$};

\end{tikzpicture}
\label{fig:simplified-annihilation-b}
}

\caption{\textbf{Two-particle annihilation, case
$\varepsilon_j = +1$: $a_j \preceq b_j$.}
(a)~A \emph{performance}: particles $I$ (thick) and $J$
(wavy) collide; both are destroyed and two ghost paths emerge
(dashed). Four distinct styles emphasize that no identity
persists through the collision.
(b)~The corresponding \emph{casting}, produced by attribution
via the far-side principle: particle~$I$ (left, thick) is
glued to the rightward ghost at~$b_j$ (thick dashed);
particle~$J$ (right, wavy) is glued to the leftward ghost
at~$a_j$ (wavy dashed). The spatial ordering is reversed.}
\label{fig:simplified-annihilation}
\end{figure}

%

\begin{figure}[t]
\centering

\subfloat[]{
\begin{tikzpicture}[scale=0.8]

\def\width{3.7}
\def\height{3}

\draw[gray, thick] (-\width/2, 0) -- (\width/2, 0);
\draw[gray, thick] (-\width/2, \height) -- (\width/2, \height);
\node[right] at (\width/2, 0) {$\mathcal{X}$};
\node[right] at (\width/2, \height) {$\mathcal{Y}$};

\coordinate (startI) at (-1.0, 0);
\coordinate (startJ) at (1.0, 0);

\coordinate (collision) at (0, 1.5);

\coordinate (endB) at (-1.0, \height);    
\coordinate (endA) at (1.0, \height);     

\draw[colA, line width=2.5pt] (startI) -- (collision);
\draw[colB, line width=1.5pt, decorate,
      decoration={snake, segment length=5pt, amplitude=1pt}]
    (startJ) -- (collision);

\draw[colGone, line width=1.5pt, dashed,
      dash pattern={on 4pt off 2pt},
      double, double distance=2pt]
    (collision) -- (endB);
\draw[colGone, line width=1.5pt, dashed,
      dash pattern={on 4pt off 2pt}]
    (collision) -- (endA);

\node[circle, fill=black, inner sep=2.5pt] at (collision) {};

\node[circle, fill=colA, inner sep=2.5pt] at (startI) {};
\node[circle, fill=colB, inner sep=2.5pt] at (startJ) {};

\node[circle, draw=colGone, fill=white, line width=1pt,
      inner sep=2pt] at (endB) {};
\node[circle, draw=colGone, fill=white, line width=1pt,
      inner sep=2pt] at (endA) {};

\node[below left] at (startI) {$I$};
\node[below right] at (startJ) {$J$};
\node[above left] at (endB) {$b_j$};
\node[above right] at (endA) {$a_j$};

\end{tikzpicture}
\label{fig:simplified-annihilation-neg-a}
}
\hfill
\subfloat[]{
\begin{tikzpicture}[scale=0.8]

\def\width{3.7}
\def\height{3}

\draw[gray, thick] (-\width/2, 0) -- (\width/2, 0);
\draw[gray, thick] (-\width/2, \height) -- (\width/2, \height);
\node[right] at (\width/2, 0) {$\mathcal{X}$};
\node[right] at (\width/2, \height) {$\mathcal{Y}$};

\coordinate (startI) at (-1.0, 0);
\coordinate (startJ) at (1.0, 0);
\coordinate (V) at (0, 1.5);
\coordinate (endB) at (-1.0, \height);    
\coordinate (endA) at (1.0, \height);     

\draw[colA, line width=2.5pt] (startI) -- (V);
\draw[colA, line width=2.5pt, dashed,
      dash pattern={on 4pt off 2pt}]
    (V) -- (endA);

\draw[colB, line width=1.5pt, decorate,
      decoration={snake, segment length=5pt, amplitude=1pt}]
    (startJ) -- (V);
\draw[colB, line width=1.5pt, dashed,
      dash pattern={on 5pt off 3pt}, decorate,
      decoration={snake, segment length=5pt, amplitude=1pt},
      double, double distance=2pt]
    (V) -- (endB);

\node[circle, fill=white, inner sep=0pt, minimum size=6pt,
      line width=0.8pt, draw=black] at (V) {};

\node[circle, fill=colA, inner sep=2.5pt] at (startI) {};
\node[circle, fill=colB, inner sep=2.5pt] at (startJ) {};

\node[circle, draw=colB, fill=white, line width=1pt,
      inner sep=2pt] at (endB) {};
\node[circle, draw=colA, fill=white, line width=1pt,
      inner sep=2pt] at (endA) {};

\node[below left] at (startI) {$I$};
\node[below right] at (startJ) {$J$};
\node[above left] at (endB) {$b_j$};
\node[above right] at (endA) {$a_j$};

\end{tikzpicture}
\label{fig:simplified-annihilation-neg-b}
}

\caption{\textbf{Two-particle annihilation, case
$\varepsilon_j = -1$: $b_j \prec a_j$.}
The companion of \Cref{fig:simplified-annihilation}: the ghost
pair is now ordered right-left, so the leftward position
is~$b_j$ (slot $(j,2)$) and the rightward is~$a_j$ (slot
$(j,1)$).
(a)~A \emph{performance}: the collision and the two anonymous
ghost paths are exactly as before; only the slot labels have
swapped sides.
(b)~The corresponding \emph{casting}: the far-side principle
again sends particle~$I$ (left, thick) to the rightward
ghost---now $a_j$ (thick dashed)---and particle~$J$ (right,
wavy) to the leftward ghost---now $b_j$ (wavy dashed). In
both cases the higher-indexed particle takes the leftward
position; what changes with the sign is which \emph{slot}
that position carries.}
\label{fig:simplified-annihilation-negative}
\end{figure}

\subsubsection{From local to global}

Attribution takes a performance as input and produces a
casting $(\pi, \paths)$. For each initial particle~$I$, follow
its path to the collision vertex (if any) and apply the
far-side principle (\Cref{principle:swap}) to determine which
ghost path it joins. The
bijection~$\pi$ combines local and global structure: the second
coordinate of $\pi(I) = (j, m)$ comes from the far-side principle
(which assigns $I$ to either the first or second path in the
ordered pair, giving $m \in \{1,2\}$), while the first
coordinate~$j$ comes from the performance's global numbering
of ghost pairs.

\begin{definition}[Attribution]\label{def:attribution}
\emph{Attribution} constructs a casting $(\pi, \paths)$ from a
performance by applying the far-side principle at each collision
vertex. For each initial particle~$I$:
\begin{itemize}
\item If $I$ collides, the far-side principle glues its incoming
  path to an outgoing ghost path, determining the endpoint
  $\pi(I) \in \Ghosts$ and the glued path~$P_I$ (at a vertex
  hosting several ghost pairs, the pairing rule below first
  decides which pair $I$ fills).
\item If $I$ survives, $\pi(I) \in \Survivors$ and $P_I$ is
  the path from~$x_I$ to the survivor position (this path
  avoids all collision vertices).
\end{itemize}
The output is a casting $(\pi, (P_I)_{I \in \Actors})$.
\end{definition}

Since both colliding particles are destroyed, each particle
participates in at most one collision. Attribution therefore
applies the far-side principle separately at each collision
vertex---no chaining arises (unlike coalescence, where the
heir continues and may collide again).

\smallskip\noindent\textbf{Pairing rule (high-indegree
vertices).}
At a collision vertex hosting several ghost pairs
(\Cref{def:collision-diagram}), attribution must additionally
decide which arriving particles fill which pair. Sort the
arriving particles $I_1 < I_2 < \cdots < I_m$ and pair them
consecutively, matching the pairs to the anchors in order:
$(I_1, I_2)$ fills the ghost pair at anchor~$(v, 1)$,
$(I_3, I_4)$ the pair at anchor~$(v, 2)$, and so on. If $m$
is odd, the last particle~$I_m$ continues along the departing
path (possibly of length zero, when a survivor ends at the
vertex). Each pair is then glued to its own anchor's two
ghost paths by the far-side principle, exactly as in the
binary case. Rehearsal processes such a vertex by the same
consecutive pairing (\Cref{sec:proof-rehearsal},
page~\pageref{high-indegree}), where we also explain why the
choice of this convention is immaterial.

\subsubsection{Example}

\begin{example}[$n = 4$, $k = 1$]\label{ex:attribution-n4}
In \Cref{fig:intro-annihilation}, particles~$2$ and~$3$ ($I = 2
< 3 = J$) collide at~$c$. The ghost pair has $a \preceq b$,
so $\varepsilon = +1$. The far-side principle glues particle~$3$
($J$, higher index) to the leftward ghost at~$a$, and
particle~$2$ ($I$, lower index) to the rightward ghost
at~$b$.
Survivors~$1$ and~$4$ pass through the collision diagram to
their final positions. The resulting casting has
$\pi(1) = 1$ (survivor slot at~$y_1$), $\pi(2) = (1,2)$
(the second ghost from the first and only ghost pair),
$\pi(3) = (1,1)$
(the first ghost from the first and only ghost pair),
$\pi(4) = 2$ (survivor slot at~$y_2$),
with glued paths:
\begin{itemize}
\item $P_1$: from $x_1$ to $y_1$ (survivor);
\item $P_2$: from $x_2$ to $c$, then to $b$ (ghost);
\item $P_3$: from $x_3$ to $c$, then to $a$ (ghost);
\item $P_4$: from $x_4$ to $y_2$ (survivor).
\end{itemize}
This is the successful casting shown in
\Cref{fig:successful-casting}.
\end{example}

\subsubsection{Candidacy emerges}

\begin{proposition}\label{prop:attribution-candidate}
Let $(\pi, \paths)$ be the casting from attributing a performance.
Then $\pi$ is a candidate.
\end{proposition}

\begin{proof}
At the collision creating ghost pair~$j$, let particles $I < J$
collide. When $a_j \preceq b_j$ (so $\varepsilon_j = +1$), the
far-side principle assigns $\pi(J) = (j,1)$ (position $a_j$, the
leftward ghost) and $\pi(I) = (j,2)$ (position~$b_j$, rightward).
Then $\pi^{-1}(j,1) = J > I = \pi^{-1}(j,2)$, matching the
candidacy condition for $\varepsilon_j = +1$. When
$b_j \prec a_j$ (so $\varepsilon_j = -1$), the far-side principle
assigns $\pi(I) = (j,1)$ (position~$a_j$, now rightward) and
$\pi(J) = (j,2)$ (position~$b_j$, leftward). Then
$\pi^{-1}(j,1) = I < J = \pi^{-1}(j,2)$, matching the
candidacy condition for $\varepsilon_j = -1$.
\end{proof}

\subsection{Rehearsal}
\label{sec:proof-rehearsal}

\emph{Rehearsal} reverses attribution: given a candidate casting
$(\pi, \paths)$, we attempt to interpret it as a performance by
scanning crossings in temporal order.

\subsubsection{The rehearsal algorithm}

\begin{enumerate}
\item Initialize: all particles are active.
\item Find the earliest crossing among active paths (ties
  broken by the lexicographically smallest pair of
  particles, making the scan deterministic). If none exist,
  success.
\item Test whether this crossing is valid (see below).
  \begin{itemize}
  \item If yes: record the collision, deactivate both particles,
    and continue scanning for the next crossing.
  \item If no: failure---the casting is ``failed.''
  \end{itemize}
\end{enumerate}
Since both participants are deactivated at each valid collision,
each particle is processed at most once.

\medskip\noindent\textbf{High-indegree vertices.}\label{high-indegree}
When $m > 2$ active particles meet at a single vertex,
sort them by initial index:
$I_1 < I_2 < \cdots < I_m$. Process consecutive pairs
$(I_1, I_2)$, $(I_3, I_4)$, etc.; each pair is tested (and,
if valid, recorded, at the vertex's next anchor) as a
separate annihilation. If $m$ is
odd, the last incoming particle~$I_m$
is not paired at this vertex and remains active. This pairing rule
is a convention: the only requirement is that attribution and
rehearsal use the same rule, so that the bijection between
successful castings and performances is preserved. Ghost
anonymity ensures that the specific convention does not affect
the annihilation weight.

\subsubsection{Valid vs.\ spurious crossings}

A crossing between particles $I < J$ is \emph{valid} if both
particles are destined for the same ghost pair:
$\{\pi(I), \pi(J)\} = \{(j,1), (j,2)\}$ for some ghost
pair~$j$---that is, the casting assigns $I$ and $J$ to the
two slots of a single twin pair. If the constraint fails
(one or both particles are destined for survivor slots, or
they map to different ghost pairs), the crossing is
\emph{spurious}.

\begin{definition}[Active set]\label{def:active-set}
At any stage of a run of rehearsal, the \emph{active set}
consists of the particles that have not been deactivated.
\end{definition}

\noindent
The active set is \emph{pair-complete}: for each ghost
pair~$j$, the two actors $\pi^{-1}(j,1)$ and $\pi^{-1}(j,2)$
are either both active or both inactive. Indeed, a valid
collision deactivates exactly the two actors assigned to one
ghost pair---we say the collision \emph{fills} that
pair---while actors destined for survivor slots are never
deactivated.

\begin{definition}[Successful casting]\label{def:successful}
A candidate casting is \emph{successful} if rehearsal finds every
crossing valid, recording an annihilation at each, and never
encounters a spurious one.
\end{definition}

\Cref{fig:castings} illustrates both outcomes. In
\subref{fig:successful-casting}, particles~$2$ and~$3$ are
assigned to the ghost pair: their crossing at $c$ is valid, and
rehearsal succeeds. In \subref{fig:failed-casting}, particles~$1$
and~$3$ are assigned to the ghost pair instead, but the first
crossing involves particles~$1$ and~$2$---a spurious crossing,
since particle~$2$ is destined for a survivor position.

\begin{remark}[Dual interpretation]\label{rem:dual-interpretation}
The rehearsal algorithm will serve two purposes:
\begin{itemize}
\item \textbf{As a classifier:} does this casting correspond
  to a performance? Success produces a performance; failure
  means it does not.
\item \textbf{As a pairing device:} a failed run halts at a
  specific spurious crossing, and swapping the two path
  suffixes there will pair the casting with another failed
  casting of opposite sign.
\end{itemize}
In \Cref{sec:proof-involution} the two purposes combine into a
single map, the involution~$\iota$: it fixes the successful
castings and matches the failed ones in mutually canceling
pairs.
\end{remark}

\subsubsection{The stall obstruction: crossings must exist}
\label{sec:proof-no-stall}

The rehearsal algorithm terminates either when it encounters a
spurious crossing (failure) or when no crossings remain among
active paths (success). A priori there is a third conceivable
outcome, a \emph{stall}. Success (\Cref{def:successful})
means only that the run never crashed on a spurious
crossing; a priori it does not guarantee that a performance
was reconstructed. A stall is a successful termination while
some ghost role is still unfilled. A stalled casting would break
the proof. Its run did not fail, so the casting would never be
paired off and canceled (\Cref{rem:dual-interpretation}); yet
it would yield no performance with the prescribed final
state---some ghost pair would have no collision to create it.
Its contribution to the determinant would remain in the signed
sum, matched by nothing.

The stall is ruled out in two steps, separating the
combinatorics from the geometry. The combinatorial core,
\Cref{lem:no-ghosts-order}, is a statement about candidate
bijections alone, making no reference to paths or crossings;
\Cref{lem:crossings-exist} then supplies the geometric input
through the order preservation of non-crossing paths
(\Cref{prop:order-preservation}). \Cref{prop:rehearsal-boundary}
records
the conclusion.

\begin{lemma}[Candidacy reverses each ghost pair]
\label{lem:no-ghosts-order}
Let $\pi$ be a candidate bijection and let $I < J$ be the two
actors of a ghost pair~$j$. Then
$y_{\pi(J)} \preceq y_{\pi(I)}$: the two targets are in weakly
\emph{reversed} spatial order.
\end{lemma}

\begin{proof}
When $\varepsilon_j = +1$, candidacy gives
$\pi^{-1}(j,1) > \pi^{-1}(j,2)$, so $\pi(J) = (j,1)$ and
$\pi(I) = (j,2)$, and the targets satisfy
$y_{\pi(J)} = a_j \preceq b_j = y_{\pi(I)}$
(\Cref{def:ghost-sign}). When $\varepsilon_j = -1$, candidacy
gives $\pi^{-1}(j,1) < \pi^{-1}(j,2)$, so $\pi(I) = (j,1)$ and
$\pi(J) = (j,2)$, and
$y_{\pi(J)} = b_j \prec a_j = y_{\pi(I)}$.
\end{proof}

\begin{corollary}[Crossings exist]\label{lem:crossings-exist}
In a candidate casting, the paths of the two actors of each
ghost pair cross.
\end{corollary}

\begin{proof}
Let $I < J$ be the two actors of ghost pair~$j$. If
$a_j = b_j$, the two paths share their endpoint vertex
and hence cross. Otherwise the endpoints are distinct:
if the paths did not cross,
\Cref{prop:order-preservation}, applied to the
two-path family $\{P_I, P_J\}$, would give
$y_{\pi(I)} \prec y_{\pi(J)}$---contradicting
\Cref{lem:no-ghosts-order}.
\end{proof}

\begin{proposition}[Successful rehearsal produces a
performance]\label{prop:rehearsal-boundary}
The output of a successful run of rehearsal on a candidate
casting $(\pi, \paths)$ is a performance with the prescribed
final state~$\FinalState$.
\end{proposition}

\begin{proof}
\emph{Every ghost pair is filled.} By pair-completeness of
the active set (\Cref{def:active-set}), both actors of each
unfilled ghost pair are still active at every stage of the
run. If some ghost pair were unfilled at termination, the
crossing of its actors' paths (\Cref{lem:crossings-exist})
would be a crossing between two active paths, and rehearsal
would not have stopped.

\medskip\noindent\emph{The final state is the prescribed one.}
At the collision filling pair~$j$, the recorded ghost paths
are the suffixes of the two participants' casting paths, which
end at the prescribed positions $a_j$ and~$b_j$ (the
participants are mapped by~$\pi$ onto the slots $(j,1)$
and~$(j,2)$, \Cref{def:casting}). The surviving particles are
never deactivated and keep their casting paths, reaching the
prescribed survivor positions. The recorded performance
therefore has $k$ collisions and the final
state~$\FinalState$.
\end{proof}

\begin{remark}[A single obstruction]\label{rem:single-obstruction}
A priori, rehearsing a casting could break down in several
ways: the pair processed at a crossing might fail to be
adjacent in the active set; a crossing might be spurious; or
the run might stall, terminating with an unfilled ghost pair.
The planarity assumptions collapse this list: the
contiguous-block property (\Cref{prop:adjacent-crossings})
keeps every processed pair adjacent, and
\Cref{prop:rehearsal-boundary}
rules out the mismatched termination. The spurious
crossing is therefore the only genuine obstruction---and it is
exactly the failure that segment swap pairs off
(\Cref{sec:proof-involution}). In the companion paper the
mismatched termination is the subtle case, requiring a rigidity
analysis of partial assignments (the assignment-rigidity lemma
of~\cite{SU2026coalescence}); here ghost pairs are filled
whole---both actors deactivate together---so the one-pair
argument above suffices.
\end{remark}

\subsubsection{Attribution and rehearsal are mutually inverse}

\begin{proposition}\label{prop:bijection}
Attribution and rehearsal are mutual inverses between performances and
successful castings.
\end{proposition}

\begin{proof}

\leavevmode

\emph{Rehearsal $\circ$ Attribution $= \id$.}
Let $\perf$ be a performance. The casting
$(\pi, \paths) = \Attribution(\perf)$ is a candidate
(\Cref{prop:attribution-candidate}). The crossings in $\paths$
are exactly the collisions in $\perf$, so rehearsal encounters
them in the same temporal order and processes each as valid
(the constraint
$\{\pi(I), \pi(J)\} = \{(j,1), (j,2)\}$ holds by
construction of attribution). Rehearsal therefore succeeds
and reconstructs~$\perf$.

\medskip

\emph{Attribution $\circ$ Rehearsal $= \id$.}
Let $(\pi, \paths)$ be a successful casting and $\perf$ the
performance produced by rehearsal---a performance with the
prescribed final state~$\FinalState$, by
\Cref{prop:rehearsal-boundary}. We verify that
$\Attribution(\perf) = (\pi, \paths)$ by checking each
collision in temporal order.

At a collision between particles $I < J$ filling ghost
pair~$j$, candidacy ensures the far-side principle assigns the
same roles as~$\pi$: if $\varepsilon_j = +1$, then
$\pi^{-1}(j,1) > \pi^{-1}(j,2)$, so the higher-indexed
particle maps to $(j,1)$ and the lower-indexed to
$(j,2)$---matching attribution. The case $\varepsilon_j = -1$
is symmetric. Since
both particles deactivate, later collisions are unaffected.
After all collisions, surviving particles retain their
original paths and endpoints. Thus attribution recovers
$(\pi, \paths)$. At a vertex bearing several ghost pairs,
both maps use the anchor order, so the correspondence
remains exact even when two pairs from that vertex carry
identical ghost paths.
\end{proof}

\subsection{Segment swap}
\label{sec:proof-swap}

Failed castings cancel via \emph{segment swap}, illustrated
in \Cref{fig:segment-swap}.

\begin{definition}[Segment swap]\label{def:segment-swap}
\emph{Segment swap} is a transformation on castings. Given a
casting $(\pi, \paths)$ where paths $P_I$ and $P_J$ cross at
vertex~$v$, segment swap at~$v$ produces a new casting
$(\pi', \paths')$:
\begin{itemize}
\item $P_I'$: prefix of $P_I$ from $x_I$ to $v$, then
  suffix of $P_J$ from $v$ to $y_{\pi(J)}$;
\item $P_J'$: prefix of $P_J$ from $x_J$ to $v$, then
  suffix of $P_I$ from $v$ to $y_{\pi(I)}$.
\end{itemize}
The bijection updates to
$\pi' = (I\; J) \circ \pi$, swapping the destinations.
\end{definition}

\begin{figure}[t]
\centering
\captionsetup[subfigure]{justification=centering}

\subfloat[Before swap: paths cross at $c$]{%
\begin{tikzpicture}[scale=0.6,
    vertex/.style={circle, fill, inner sep=1.5pt}]

\draw[gray, thick] (0, 0) -- (6, 0);
\draw[gray, thick] (0, 4) -- (6, 4);
\node[right] at (6.2, 0) {\small $t=0$};
\node[right] at (6.2, 4) {\small $t=T$};

\node[vertex, colA] (x1) at (1,0) {};
\node[below] at (x1) {$x_1$};
\node[vertex, colB] (x2) at (3,0) {};
\node[below] at (x2) {$x_2$};

\node[vertex, colA] (y1) at (2,4) {};
\node[above] at (y1) {$y_1$};
\node[vertex, colB] (y2) at (5,4) {};
\node[above] at (y2) {$y_2$};

\node[vertex, black] (v) at (3,2) {};
\node[right=0.2] at (v) {$c$};

\draw[particleA] (x1) -- (2,1) -- (v);
\draw[particleA] (v) -- (y1);

\draw[particleB] (x2) -- (v);
\draw[particleB] (v) -- (4,3) -- (y2);

\begin{scope}[shift={(7,1)}]
  \draw[particleA] (0,1.2) -- (0.7,1.2);
  \node[right] at (0.8,1.2) {\small $P_1$};
  \draw[particleB] (0,0.4) -- (0.7,0.4);
  \node[right] at (0.8,0.4) {\small $P_2$};
\end{scope}

\end{tikzpicture}%
\label{fig:swap-before}%
}
\hfill
\subfloat[After swap: endpoints exchanged]{%
\begin{tikzpicture}[scale=0.6,
    vertex/.style={circle, fill, inner sep=1.5pt}]

\draw[gray, thick] (0, 0) -- (6, 0);
\draw[gray, thick] (0, 4) -- (6, 4);
\node[right] at (6.2, 0) {\small $t=0$};
\node[right] at (6.2, 4) {\small $t=T$};

\node[vertex, colA] (x1) at (1,0) {};
\node[below] at (x1) {$x_1$};
\node[vertex, colB] (x2) at (3,0) {};
\node[below] at (x2) {$x_2$};

\node[vertex, colB] (y1) at (2,4) {};
\node[above] at (y1) {$y_1$};
\node[vertex, colA] (y2) at (5,4) {};
\node[above] at (y2) {$y_2$};

\node[vertex, black] (v) at (3,2) {};
\node[right=0.2] at (v) {$c$};

\draw[particleA] (x1) -- (2,1) -- (v);
\draw[particleA] (v) -- (4,3) -- (y2);

\draw[particleB] (x2) -- (v);
\draw[particleB] (v) -- (y1);

\begin{scope}[shift={(7,1)}]
  \draw[particleA] (0,1.2) -- (0.7,1.2);
  \node[right] at (0.8,1.2) {\small $P'_1$};
  \draw[particleB] (0,0.4) -- (0.7,0.4);
  \node[right] at (0.8,0.4) {\small $P'_2$};
\end{scope}

\end{tikzpicture}%
\label{fig:swap-after}%
}
\caption{The segment swap operation. (a)~Paths
$P_1$ (solid) and $P_2$ (double) cross at vertex $c$. (b)~After the
swap, final segments are exchanged: $P'_1$ follows $P_1$ to $c$, then
$P_2$'s tail to $y_2$; $P'_2$ follows $P_2$ to $c$, then $P_1$'s tail
to $y_1$. The paths still cross at $c$, but now go to swapped endpoints.}
\label{fig:segment-swap}
\end{figure}

\begin{lemma}[Swap properties]\label{lem:swap-properties}
Segment swap is:
\begin{enumerate}[label=(\roman*),
                   ref=\roman*]
\item \label{itm:swap-involution}
  an involution: swapping twice recovers the original;
\item \label{itm:swap-weight}
  weight-preserving: $w(P_I) \cdot w(P_J)
  = w(P_I') \cdot w(P_J')$;
\item \label{itm:swap-sign}
  sign-reversing: $\sgn(\pi') = -\sgn(\pi)$.
\end{enumerate}
\end{lemma}

\begin{proof}
\eqref{itm:swap-involution} After swapping at $v$, the
paths $P_I'$ and $P_J'$ still cross at $v$. Swapping again
restores the original suffixes.

\eqref{itm:swap-weight} Each edge appears in exactly one
path before and after the swap, so the product of weights
is unchanged.

\eqref{itm:swap-sign} The bijection
$\pi' = (I\;J) \circ \pi$ differs from $\pi$ by a
transposition, so $\sgn(\pi') = -\sgn(\pi)$.
\end{proof}

\Cref{fig:involution-pair} shows the pairing in action on
the failed casting from \Cref{fig:failed-casting}. Segment swap
at the spurious crossing exchanges the suffixes of particles~$1$
and~$2$, producing a paired failed casting with reversed
permutation sign but identical weight. The crossing persists
after the swap and remains spurious. One point requires care:
the swap could conceivably create a new crossing at some
\emph{earlier} vertex, derailing the second run before it
returns to this one. It cannot: the swap alters only the two
suffixes after the crossing vertex, and every vertex of a
swapped suffix follows the crossing vertex in the time order,
so nothing changes at or before it. The second run therefore
halts at the same crossing, and a second swap restores the
original casting: the two failed castings form a mutually
canceling pair. The proof of \Cref{thm:involution} makes this
argument precise.

%
%

\begin{figure}[t]
\centering
\captionsetup[subfigure]{justification=centering}

\subfloat[Failed casting $\pi$]{%
\begin{tikzpicture}[scale=0.55]
  \begin{scope}
    \clip (-0.2,-0.2) rectangle (11.2,6.2);

    \foreach \x in {-1,...,12} {
      \foreach \t in {-1,...,7} {
        \pgfmathparse{mod(\x+\t,2)==0 ? 1 : 0}
        \ifnum\pgfmathresult>0
          \fill[gray!30] (\x,\t) circle (1.5pt);
        \fi
      }
    }

    \foreach \x in {-2,0,2,4,6,8,10,12} {
      \foreach \t in {-2,0,2,4,6,8} {
        \draw[gray!30, thin] (\x,\t) -- (\x-1,\t+1);
        \draw[gray!30, thin] (\x,\t) -- (\x+1,\t+1);
      }
    }
    \foreach \x in {-1,1,3,5,7,9,11} {
      \foreach \t in {-1,1,3,5,7} {
        \draw[gray!30, thin] (\x,\t) -- (\x-1,\t+1);
        \draw[gray!30, thin] (\x,\t) -- (\x+1,\t+1);
      }
    }
  \end{scope}

  \draw[->, thick] (-0.3,0) -- (11,0) node[right] {$x$};
  \draw[->, thick] (0,-0.3) -- (0,6.7) node[above] {$t$};

  \draw (-0.1, 6) -- (0.1, 6);
  \node[left] at (-0.15,6) {\small $T$};

  \coordinate (spurious) at (5,1);


  \draw[particleC] (8,0) -- (7,1) -- (6,2) -- (5,3) -- (4,4)
    -- (3,5) -- (2,6);

  \draw[particleA] (4,0) -- (5,1) -- (6,2) -- (7,3) -- (8,4)
    -- (9,5) -- (8,6);

  \draw[particleB] (6,0) -- (5,1) -- (4,2) -- (3,3)
    -- (2,4) -- (3,5) -- (4,6);

  \draw[particleD] (10,0) -- (9,1) -- (10,2) -- (9,3)
    -- (10,4) -- (9,5) -- (10,6);

  \draw[red, line width=2pt] (spurious) circle (0.4);
  \node[font=\large\bfseries, red, above right=-0.05
    and 0.15] at (spurious) {!};

  \fill[colA] (4,0) circle (4pt);
  \fill[colB] (6,0) circle (4pt);
  \fill[colC] (8,0) circle (4pt);
  \fill[colD] (10,0) circle (4pt);

  \node[below] at (4,-0.15) {\small $x_1$};
  \node[below] at (6,-0.15) {\small $x_2$};
  \node[below] at (8,-0.15) {\small $x_3$};
  \node[below] at (10,-0.15) {\small $x_4$};

  \draw[colGone, fill=white, line width=1.5pt] (2,6) circle (4pt);
  \fill[colB] (4,6) circle (4pt);
  \draw[colGone, fill=white, line width=1.5pt] (8,6) circle (4pt);
  \fill[colD] (10,6) circle (4pt);

  \node[above] at (2,6.15) {\small $a$};
  \node[above] at (4,6.15) {\small $y_1$};
  \node[above] at (8,6.15) {\small $b$};
  \node[above] at (10,6.15) {\small $y_2$};

  \begin{scope}[shift={(11.5,0.2)}]
    \draw[particleA] (0,5.6) -- (0.7,5.6);
    \node[right, align=left, inner sep=1pt] at (0.8,5.6)
      {\small $1 \to b$};
    \draw[particleB] (0,4.4) -- (0.7,4.4);
    \node[right, align=left, inner sep=1pt] at (0.8,4.4)
      {\small $2 \to y_1$};
    \draw[particleC] (0,3.2) -- (0.7,3.2);
    \node[right, align=left, inner sep=1pt] at (0.8,3.2)
      {\small $3 \to a$};
    \draw[particleD] (0,2.0) -- (0.7,2.0);
    \node[right, align=left, inner sep=1pt] at (0.8,2.0)
      {\small $4 \to y_2$};
  \end{scope}

\end{tikzpicture}%
\label{fig:involution-before}%
}

\medskip

\subfloat[Paired failed casting $\pi'=(1\;2)\circ\pi$]{%
\begin{tikzpicture}[scale=0.55]
  \begin{scope}
    \clip (-0.2,-0.2) rectangle (11.2,6.2);

    \foreach \x in {-1,...,12} {
      \foreach \t in {-1,...,7} {
        \pgfmathparse{mod(\x+\t,2)==0 ? 1 : 0}
        \ifnum\pgfmathresult>0
          \fill[gray!30] (\x,\t) circle (1.5pt);
        \fi
      }
    }

    \foreach \x in {-2,0,2,4,6,8,10,12} {
      \foreach \t in {-2,0,2,4,6,8} {
        \draw[gray!30, thin] (\x,\t) -- (\x-1,\t+1);
        \draw[gray!30, thin] (\x,\t) -- (\x+1,\t+1);
      }
    }
    \foreach \x in {-1,1,3,5,7,9,11} {
      \foreach \t in {-1,1,3,5,7} {
        \draw[gray!30, thin] (\x,\t) -- (\x-1,\t+1);
        \draw[gray!30, thin] (\x,\t) -- (\x+1,\t+1);
      }
    }
  \end{scope}

  \draw[->, thick] (-0.3,0) -- (11,0) node[right] {$x$};
  \draw[->, thick] (0,-0.3) -- (0,6.7) node[above] {$t$};

  \draw (-0.1, 6) -- (0.1, 6);
  \node[left] at (-0.15,6) {\small $T$};

  \coordinate (spurious) at (5,1);


  \draw[particleC] (8,0) -- (7,1) -- (6,2) -- (5,3) -- (4,4)
    -- (3,5) -- (2,6);

  \draw[particleA] (4,0) -- (5,1) -- (4,2) -- (3,3)
    -- (2,4) -- (3,5) -- (4,6);

  \draw[particleB] (6,0) -- (5,1) -- (6,2) -- (7,3) -- (8,4)
    -- (9,5) -- (8,6);

  \draw[particleD] (10,0) -- (9,1) -- (10,2) -- (9,3)
    -- (10,4) -- (9,5) -- (10,6);

  \draw[red, line width=2pt] (spurious) circle (0.4);
  \node[font=\large\bfseries, red, above right=-0.05
    and 0.15] at (spurious) {!};

  \fill[colA] (4,0) circle (4pt);
  \fill[colB] (6,0) circle (4pt);
  \fill[colC] (8,0) circle (4pt);
  \fill[colD] (10,0) circle (4pt);

  \node[below] at (4,-0.15) {\small $x_1$};
  \node[below] at (6,-0.15) {\small $x_2$};
  \node[below] at (8,-0.15) {\small $x_3$};
  \node[below] at (10,-0.15) {\small $x_4$};

  \draw[colGone, fill=white, line width=1.5pt] (2,6) circle (4pt);
  \fill[colA] (4,6) circle (4pt);
  \draw[colGone, fill=white, line width=1.5pt] (8,6) circle (4pt);
  \fill[colD] (10,6) circle (4pt);

  \node[above] at (2,6.15) {\small $a$};
  \node[above] at (4,6.15) {\small $y_1$};
  \node[above] at (8,6.15) {\small $b$};
  \node[above] at (10,6.15) {\small $y_2$};

  \begin{scope}[shift={(11.5,0.2)}]
    \draw[particleA] (0,5.6) -- (0.7,5.6);
    \node[right, align=left, inner sep=1pt] at (0.8,5.6)
      {\small $1 \to y_1$};
    \draw[particleB] (0,4.4) -- (0.7,4.4);
    \node[right, align=left, inner sep=1pt] at (0.8,4.4)
      {\small $2 \to b$};
    \draw[particleC] (0,3.2) -- (0.7,3.2);
    \node[right, align=left, inner sep=1pt] at (0.8,3.2)
      {\small $3 \to a$};
    \draw[particleD] (0,2.0) -- (0.7,2.0);
    \node[right, align=left, inner sep=1pt] at (0.8,2.0)
      {\small $4 \to y_2$};
  \end{scope}

\end{tikzpicture}%
\label{fig:involution-after}%
}

\caption{The sign-reversing involution in action: a matched
pair of failed castings related by segment swap at the
spurious crossing.
(a)~Particles~$1$ and~$2$ meet at the circled vertex~$v$, but
particle~$2$ is destined for a survivor position---spurious
crossing.
(b)~After the swap, suffixes are exchanged: particle~$1$ now
reaches~$y_1$, particle~$2$ now reaches~$b$.
The crossing persists and is still spurious, so a second
swap restores~(a). The permutation sign reverses
($\sgn(\pi') = -\sgn(\pi)$) while the weight is preserved,
so the two castings cancel.}
\label{fig:involution-pair}

\end{figure}

\subsection{The sign-reversing involution}
\label{sec:proof-involution}

\subsubsection{The swap criterion}

Segment swap of an arbitrary candidate casting need not
itself be a candidate casting. The next lemma pins down
exactly when it is, at the crossings that rehearsal tests.

\begin{lemma}[Swap criterion]\label{lem:key-lemma}
At an iteration of rehearsal on a candidate casting
$(\pi, \paths)$, let $I < J$ be the two active particles
whose crossing is being tested. The segment swap at this
crossing preserves candidacy if and only if the crossing is
spurious: $\{\pi(I), \pi(J)\} \neq \{(j,1), (j,2)\}$ for
every ghost pair~$j$.
\end{lemma}

\begin{proof}
The swap exchanges the roles of $I$ and $J$: the new
bijection $\pi'$ satisfies $\pi'(I) = \pi(J)$ and
$\pi'(J) = \pi(I)$.

\medskip

\emph{Valid crossing
($\{\pi(I), \pi(J)\} = \{(j,1), (j,2)\}$).}
Before swap, candidacy requires $\pi^{-1}(j,1)$ and
$\pi^{-1}(j,2)$ to be ordered consistently with
$\varepsilon_j$. After swap, their ordering is reversed.
Candidacy is violated.

\medskip

\emph{Spurious crossing
($\{\pi(I), \pi(J)\} \neq \{(j,1), (j,2)\}$ for all $j$).}
We verify that $\pi'$ is still a candidate. First, the
tested particles are adjacent in the active set: the active
particles at the crossing vertex form a contiguous block of
the active set (\Cref{prop:adjacent-crossings}), and the
pairs tested there are consecutive in the block (the
high-indegree rule, \Cref{sec:proof-rehearsal}). Now consider
a ghost pair~$j$. If both slots $(j,1), (j,2)$ are outside
$\{\pi(I), \pi(J)\}$, the swap does not affect pair~$j$.
Otherwise, exactly one slot belongs to $I$ or $J$, and the
other belongs to some third actor~$K$. By pair-completeness
of the active set (\Cref{def:active-set}), $K$ is active: the
other slot of pair~$j$ is assigned to $I$ or~$J$, which is
active. Since $I$ and $J$ are adjacent in the active set,
$K$ lies outside the interval $[I, J]$: either $K < I < J$ or
$I < J < K$. In both cases, swapping $I$ and $J$ preserves
the ordering of $K$ relative to the participant of pair~$j$,
so candidacy for pair~$j$ is maintained.
\end{proof}

\subsubsection{The involution}

Combining rehearsal (the classifier) with segment swap (the
pairing mechanism) defines the involution on candidate
castings.

\begin{definition}[The involution $\iota$]
\label{def:global-involution}
For each candidate casting~$\casting$, set
\[
  \iota(\casting) =
    \begin{cases}
      \casting
        & \text{if rehearsal on $\casting$ succeeds,} \\[1ex]
      \parbox[t]{0.4\textwidth}{segment swap of $\casting$
        at the spurious crossing where rehearsal fails}
        & \text{otherwise.}
    \end{cases}
\]
\end{definition}

The theorem below verifies that $\iota$ is well
defined---the swap branch again produces a candidate
casting---and establishes its properties.

\subsubsection{The involution theorem}

\begin{theorem}[Involution theorem]\label{thm:involution}
The map $\iota$ is a weight-preserving, sign-reversing involution on
candidate castings. Its fixed points are exactly the successful castings.
\end{theorem}

\begin{proof}
\emph{Well-defined on candidate castings.}
On a successful casting, $\iota$ acts as the identity, which
preserves candidacy trivially. On a failed casting,
rehearsal halts at a tested crossing
that is spurious, so the swap criterion
(\Cref{lem:key-lemma}) shows that the segment swap preserves
candidacy: $\iota(\casting)$ is again a candidate casting.

\medskip

\emph{Fixed points.}
By construction, $\iota(\casting) = \casting$ if and only if
rehearsal succeeds on~$\casting$, if and only if $\casting$
is successful (\Cref{def:successful}).

\medskip

\emph{Involution.}
On a successful casting $\iota$ is the identity, so
$\iota^2 = \id$ holds trivially. Suppose instead that
rehearsal on~$\casting$ fails at the spurious crossing of
$P_I$ and~$P_J$ at vertex~$v$, and write
$\casting' = \iota(\casting)$. Note that $v$ is the
\emph{first} crossing of $P_I$ and~$P_J$: both particles are
active when rehearsal reaches~$v$, so an earlier crossing of
their paths would have halted the scan or deactivated them
sooner. Since all valid collisions before~$v$ deactivated
both of their participants, neither involved $I$ or~$J$, and
the swap alters only the suffixes of $P_I$ and~$P_J$
after~$v$: every vertex of a swapped suffix is reachable
from~$v$ and hence follows~$v$ in the linear extension
(\Cref{def:spacetime-graph}). Consequently the two
castings agree on every vertex up to and including~$v$:
no crossing at or before~$v$ appears or disappears, and
since the pre-$v$ collisions never involve $I$ or~$J$
their verdicts are unchanged. Rehearsal on~$\casting'$
therefore processes the same valid collisions in the same
order before selecting the same pair $(I,J)$ at~$v$ (the
tie-breaking rule sees the same candidates). The validity test depends
only on $\{\pi(I), \pi(J)\}$, which equals
$\{\pi'(I), \pi'(J)\}$---the same unordered set. The
crossing is still spurious, so rehearsal on~$\casting'$
fails there, and $\iota(\casting')$ is the segment swap
of~$\casting'$ at~$v$, which restores~$\casting$
(\Cref{lem:swap-properties}\eqref{itm:swap-involution}):
$\iota^2 = \id$.

\medskip

\emph{Weight preservation and sign reversal.}
These follow from \Cref{lem:swap-properties}: the swap preserves
total path weight and reverses permutation sign.
\end{proof}

\subsection{The sign identity}
\label{sec:proof-sign}

\begin{proposition}[Sign identity]\label{prop:sign-identity}
For any successful casting $(\pi, \paths)$:
\[
\sgn(\pi) = (-1)^{\#\{j : \varepsilon_j = +1\}}.
\]
\end{proposition}

\begin{proof}
Each collision fills one ghost pair. Consider the collision
filling ghost pair~$j$: particles $I < J$ annihilate. They are
adjacent in the active set at the time of the collision: the
active particles at the collision vertex form a contiguous
block of the active set (\Cref{prop:adjacent-crossings}), and
the processed pairs are consecutive in the block
(\Cref{sec:proof-rehearsal}).

The roles $(j,1)$ and $(j,2)$ are consecutive in the role
order (\Cref{def:role-order}), and the actors have $I < J$.
By the far-side principle:
\begin{itemize}
\item When $\varepsilon_j = +1$: $\pi(J) = (j,1)$ and
  $\pi(I) = (j,2)$. The smaller actor maps to the larger role
  and vice versa: one inversion, contributing $-1$.
\item When $\varepsilon_j = -1$: $\pi(I) = (j,1)$ and
  $\pi(J) = (j,2)$. Order is preserved: no inversion,
  contributing $+1$.
\end{itemize}
It remains to check that all other inversions of~$\pi$
(pairs of actors whose index order disagrees with the role
order of their assigned roles, \Cref{def:role-order}) cancel
in pairs, so that the within-pair contributions above determine
the sign.

\smallskip\noindent
\emph{Two ghost pairs.} Let $\{I, J\}$ and $\{I', J'\}$ be
the actor pairs of two distinct ghost pairs. The two index
pairs are non-crossing: at the earlier of the two
collisions, both members of the other pair are still
active, and adjacency in the active set excludes an
interleaved pattern $I < I' < J < J'$. (Simultaneous
collisions at a high-indegree vertex are decomposed into
consecutive binary annihilations,
\Cref{sec:proof-rehearsal}, and are likewise
non-crossing.) In the role order (\Cref{def:role-order})
all roles of one ghost pair precede all
roles of the other, so the four cross comparisons agree in
direction when the index pairs are disjoint (contributing
$0$ or~$4$ inversions) and split two against two when they
are nested (contributing exactly~$2$). Either way the
count is even.

\smallskip\noindent
\emph{A ghost pair and a survivor.} Let $\{I, J\}$ with
$I < J$ collide and let $S$ be a surviving actor. Survivors
never deactivate, so no survivor index lies strictly
between $I$ and~$J$ (otherwise $I$ and~$J$ would not be
adjacent in the active set): either $S < I$ or $S > J$.
Survivor slots precede ghost slots, so the two comparisons
of $S$ against $I$ and~$J$ contribute $0$ inversions if
$S < I$ and $2$ if $S > J$---again even.

\smallskip\noindent
\emph{Two survivors.} In a successful casting the survivor
paths are pairwise non-crossing (a crossing between two
active survivor paths would be spurious), so by
\Cref{prop:order-preservation} the surviving actors map to
survivor slots in spatial order: no inversions.

\smallskip
The parity of the total inversion count therefore equals
the parity of the number of within-pair inversions, one for
each ghost pair with $\varepsilon_j = +1$:
$\sgn(\pi) = (-1)^{\#\{j : \varepsilon_j = +1\}}$.
\end{proof}

\subsection{Completing the proof}
\label{sec:proof-completion}

\begin{proof}[Proof of \Cref{thm:annihilation}]
The evaluation of the formal variables~\eqref{eq:formal}
annihilates all non-candidate terms of the Leibniz expansion,
leaving the candidate bijections. By \Cref{thm:involution},
the involution $\iota$
partitions candidate castings into:
\begin{itemize}
\item \emph{Fixed points}: successful castings;
\item \emph{Matched pairs}: failed castings paired by segment swap.
\end{itemize}
Each matched pair consists of castings $(\pi, \paths)$ and
$(\pi', \paths')$ with equal path products
$\prod_I w(P_I) = \prod_I w(P_I')$
(\Cref{lem:swap-properties}\eqref{itm:swap-weight})
and $\sgn(\pi') = -\sgn(\pi)$
(\Cref{lem:swap-properties}\eqref{itm:swap-sign}). The
casting weights are therefore equal and their
contributions to the determinant cancel.

Only successful castings survive. By \Cref{prop:bijection}, these
biject with performances. We verify that each contributes
$+1 \cdot w(\casting)$ to $Z$:

\medskip

\emph{Weight:} For a successful casting $(\pi, \paths)$, the
path product $\prod_I w(P_I)$ decomposes into the same edges
as in the corresponding performance: collision diagram edges
(both incoming segments to collision vertices and survivor
paths), plus the ghost paths $\Gamma_{j,1}$, $\Gamma_{j,2}$
for each pair~$j$. The bijection is therefore
weight-preserving.

\medskip

\emph{Sign cancellation:} Each ghost pair~$j$ contributes two
sign factors. First, the Leibniz expansion contributes a factor
to $\sgn(\pi)$; by \Cref{prop:sign-identity}, this factor is
$-1$ when $\varepsilon_j = +1$ and $+1$ when $\varepsilon_j = -1$.
Second, the formal variable relations (\Cref{eq:formal}) contribute
the sign $-\varepsilon_j$: candidacy with $\varepsilon_j = +1$ requires
$\pi^{-1}(j,1) > \pi^{-1}(j,2)$, producing a factor
of $-1$; candidacy with $\varepsilon_j = -1$ requires
$\pi^{-1}(j,1) < \pi^{-1}(j,2)$, producing a factor
of $+1$. In both cases, the two factors are identical and their
product is $(-1)^2 = +1$ or $(+1)^2 = +1$. Thus the net sign
contribution from each ghost pair is $+1$, and the total
contribution of each successful casting is $+1 \cdot w(\casting)$.

\medskip

\emph{Conclusion:} By \Cref{prop:bijection}, successful
castings and performances are in one-to-one correspondence.
Since both have weight $\prod_I w(P_I)$, the
bijection preserves weight:
\[
\det(M)
= \sum_{\text{successful}} w(\casting)
= \sum_{\perf} w(\perf) = Z. 
\qedhere
\]
\end{proof}

\section{Continuous processes}\label{sec:continuous}

The annihilation formula (\Cref{thm:annihilation}) was proved
for discrete spacetime graphs. The proof, however, is purely
combinatorial: candidacy, segment swap, and the sign-reversing
involution use only the order structure of the trajectories,
never the discreteness of space or time. The companion
paper~\cite[Section~\emph{Continuous
processes}]{SU2026coalescence} develops a transfer scheme that
converts exactly this kind of discrete argument into a
measure-theoretic statement about continuous processes,
carrying the coalescence formula to Brownian motion and to
birth-death chains. This section states the continuous version
of the annihilation formula and gives the dictionary under
which the same transfer scheme applies to the annihilation
involution; the measure-theoretic construction itself is not
repeated here.

\subsection{Karlin--McGregor assumptions}
\label{sec:continuous-km}

Let $\Omega$ denote the probability space of $n$
non-interacting particles with initial positions
$x_1 \leq \cdots \leq x_n$, and write
$\mathbf{X}_T(\omega) = (X_T^1(\omega), \ldots, X_T^n(\omega))$
for the final positions. The \emph{Karlin--McGregor
assumptions}~\cite{KM1959} are:
\begin{enumerate}[label=\textup{(KM\arabic*)},
                  ref=\textup{KM\arabic*}, leftmargin=*]
\item \label{itm:km-markov}\emph{Strong Markov property}: the
  $n$-particle system is strong Markov with respect to its
  joint filtration;
\item \label{itm:km-identical}\emph{Identical, independent
  dynamics}: the particles are independent ($\PP$ is the
  product law), each following the same Markov process,
  differing only in initial position;
\item \label{itm:km-order}\emph{Order preservation}: adjacent
  particles cannot change their relative order without first
  occupying the same state;
\item \label{itm:km-meeting}\emph{Meeting times are stopping
  times}: for particles $I < J$, the first meeting time
  $\tau_{I,J} = \inf\{t : X^I_t = X^J_t\}$ is a stopping
  time.
\end{enumerate}
These assumptions hold for Brownian motion and other
continuous-path diffusions, and for skip-free birth-death
chains; see~\cite[Section~\emph{Continuous
processes}]{SU2026coalescence} for a discussion, including the
joint strong Markov property, which holds whenever the
$n$ particles are independent Feller processes.

\subsection{The finite annihilating system}
\label{sec:continuous-system}

In continuous time there is no smallest time step at which to
resolve a collision, so the annihilating system itself calls
for a construction; as in the companion paper, \emph{retention}
makes it immediate. Run the $n$ trajectories of~$\Omega$. At
the first instant at which two active particles occupy a
common state---by order preservation~(\ref{itm:km-order})
necessarily two \emph{adjacent} active particles---both
retire, and their two continuing trajectories become the ghost
pair born at that collision: no path is ever stopped. Should
three or more particles meet at a single instant, the meeting
is resolved into pairwise annihilations exactly as at a
collision vertex with several arriving paths in the discrete
model (\Cref{def:collision-diagram}). Each collision retires
two active particles, so at most $\lfloor n/2 \rfloor$
collisions occur; meeting times are stopping
times~(\ref{itm:km-meeting}), and the strong Markov
property~(\ref{itm:km-markov}) makes the evolution after each
collision a fresh instance of the same dynamics, so the
recursion is well defined.

All $n$ trajectories survive intact, and annihilation is a
relabeling of which coordinates are read as survivors and
which as ghosts. Consequently the annihilating law
$\PP_{\mathrm{int}}$ is the pushforward of the non-interacting
product law~$\PP$ under this deterministic relabeling,
followed by the independent uniform numbering of the ghost
pairs (\Cref{def:performance}); its existence and
measurability reduce to those of the coalescing law in the
companion paper~\cite[Section~\emph{Continuous
processes}]{SU2026coalescence}, the relabeling being a
measurable functional of the almost surely finite collision
structure.

\subsection{The continuous annihilation formula}
\label{sec:continuous-formula}

The final state of the annihilating system consists of the
number~$k$ of annihilations, the survivor positions, and the
positions of the numbered ghost pairs. As in the discrete
setting, the two slots of a ghost pair are pure bookkeeping:
for distinct positions, the ghost sign
(\Cref{def:ghost-sign}) determines which position occupies
slot~$(j,1)$. Accordingly, events are described by sets of
final positions compatible with one sign vector at a time.

\begin{definition}[$\varepsilon$-admissible sets]
\label{def:admissible-set}
Fix the number~$k$ of ghost pairs, hence the role
set~$\Roles$, and a sign vector
$\varepsilon \in \{+1, -1\}^k$. An
\emph{$\varepsilon$-admissible set} is a measurable set
$A \subseteq \RR^{\Roles}$ of final positions such that every
$(y_f)_{f \in \Roles} \in A$ satisfies:
\begin{itemize}
\item \emph{ghost-pair signs:} for each pair~$j$,
  $a_j \leq b_j$ when $\varepsilon_j = +1$, and $b_j < a_j$
  when $\varepsilon_j = -1$;
\item \emph{survivor order:} $y_1 \leq \cdots \leq y_s$.
\end{itemize}
As in \Cref{def:ghost-sign}, the weak inequality assigns a tie
$a_j = b_j$ to $\varepsilon_j = +1$.
\end{definition}

The survivor-order condition is no restriction: order
preservation~(\ref{itm:km-order}) keeps the surviving
particles in this order (\Cref{rem:formula-edge-cases}).

\begin{theorem}[Continuous annihilation formula]
\label{thm:continuous-annihilation}
Let the underlying process satisfy the Karlin--McGregor
assumptions
\textup{(\ref{itm:km-markov})}--\textup{(\ref{itm:km-meeting})}.
For any $\varepsilon$-admissible set~$A$:
\[
\PP_{\mathrm{int}}
  \bigl(\text{$k$ annihilations, final positions} \in A\bigr)
= \frac{1}{k!}
  \prod_{j=1}^{k} (-\varepsilon_j)
  \sum_{\pi \in \Candidates}
  \sgn(\pi)\,
  \PP\bigl(\mathbf{X}_T \in A_\pi\bigr),
\]
where $\Candidates$ is the set of candidate bijections
(\Cref{def:candidate}), $\PP$ is the law of the $n$
non-interacting particles, and
\[
A_\pi = \bigl\{(y_{\pi(1)}, \ldots, y_{\pi(n)}) :
  (y_1, \ldots, y_n) \in A\bigr\}
\]
is the set~$A$ with coordinates permuted by~$\pi$.
\end{theorem}

\noindent
The right-hand side is the Leibniz expansion of
$\det(M)$ from \Cref{thm:annihilation}, divided by the $k!$ of
the uniform ghost-pair numbering and read as a
statement about measures: the candidate set and the signs
depend only on~$\varepsilon$ and~$\pi$
(\Cref{sec:proof-castings}), never on the positions in~$A$, so
the sum is a well-defined finite signed combination of
non-interacting probabilities. On a discrete state space,
summing the discrete formula over the final states in~$A$
recovers exactly this identity.

\begin{remark}[Scope]\label{rem:continuous-scope}
The theorem is stated for one number of
annihilations~$k$ and one sign vector~$\varepsilon$ at a time.
An arbitrary event on the final state is recovered by
partitioning it according to $k$ and~$\varepsilon$---on the
almost surely full region where the survivors are sorted
(\Cref{def:admissible-set})---and summing, so this one identity
determines the entire law of the final state of the
annihilating system. In particular, the
complete-annihilation probabilities behind the Pfaffian
formulas of \Cref{sec:pfaffian} are instances with $s = 0$.
\end{remark}

\subsection{Reduction to the discrete proof}
\label{sec:continuous-proof}

The proof of \Cref{thm:continuous-annihilation} is the
transfer scheme of~\cite[Section~\emph{Continuous
processes}]{SU2026coalescence}, run on the annihilation
involution of \Cref{sec:proof} in place of the coalescence
involution. We give the dictionary and refer to the companion
paper for the measure-theoretic details, which apply verbatim.

\begin{enumerate}[label=(\roman*)]
\item \emph{Lazy rehearsal.} Rehearsal
  (\Cref{sec:proof-rehearsal}) sweeps the spacetime vertices
  in a linear order, acting only at crossings; its lazy form
  jumps directly to the first meeting of two active adjacent
  representatives. The lazy form enumerates no vertices, so
  it survives the passage to continuous spacetime unchanged.
\item \emph{Finitely many combinatorial cells.} Each
  annihilation retires two active representatives, so
  rehearsal makes at most $\lfloor n/2 \rfloor$ decisions.
  Grouping the pairs $(\pi, \omega)$ of a candidate bijection
  and an outcome by the combinatorial type of the run---which
  adjacent pairs meet, in which order, with which
  verdicts---partitions the casting space into finitely many
  cells.
\item \emph{The discrete identity, cell by cell.} On each
  cell the candidate~$\pi$, the verdicts, and the sign
  $\sgn(\pi) \prod_j (-\varepsilon_j)$ are constant. The
  discrete argument
  (\Cref{sec:proof-castings,sec:proof-completion}) uses only
  these order-theoretic data, never the numerical positions,
  so it applies to each cell unchanged.
\item \emph{Measurability.} The decisions of lazy rehearsal
  form a chain of stopping times: the first meeting of two
  active adjacent representatives is a minimum of meeting
  times~(\ref{itm:km-meeting}) over a finite measurable
  family, and the strong Markov
  property~(\ref{itm:km-markov}) restarts the system after
  each decision. This is word for word the stopping-time
  bookkeeping of the companion paper.
\item \emph{The measure-preserving swap.} The involution acts
  by segment swap at the first spurious crossing, where the
  two representatives occupy a common state. Independent,
  identical dynamics~(\ref{itm:km-identical}) make the
  post-meeting segments exchangeable, and the strong Markov
  property~(\ref{itm:km-markov}) detaches them from the past,
  so the swap preserves the measure while reversing
  $\sgn(\pi)$. Failed cells cancel in pairs; only successful
  castings remain, in bijection with the performances of the
  annihilating system, each contributing with the positive
  sign (\Cref{sec:proof-completion}).
\end{enumerate}

Summing the surviving contributions cell by cell yields
\Cref{thm:continuous-annihilation}, exactly as the companion
paper's transfer yields the continuous coalescence formula.

\section{Pairwise coalescence and Pfaffians}\label{sec:pfaffian}

This section proves that the total weight of \emph{pairwise
coalescence} is a Pfaffian of pairwise quantities, by converting the
coalescence problem to a complete annihilation problem. The
annihilation formula (\Cref{thm:annihilation})---developed for
$A + A \to \emptyset$---thus becomes a tool for proving a
result about $A + A \to A$. The conversion is the cancellative
labeling, which reinterprets coalescence events as annihilations.

\subsection{The coalescence model}
\label{sec:pfaffian-coalescence}

Before deriving the Pfaffian formula, we introduce the coalescence model
from the companion paper~\cite{SU2026coalescence}. We provide
only the definitions necessary for the cancellative labeling.

\subsubsection{Coalescence dynamics}

In the coalescence model ($A + A \to A$), when two particles collide
one \emph{heir} and one \emph{ghost} emerge: the heir continues as a
visible particle while the ghost drifts as an invisible walker that no
longer interacts with anything. This contrasts with annihilation
($A + A \to \emptyset$), where both are destroyed and a ghost pair
emerges.

\begin{definition}[Multiplicity]\label{def:multiplicity}
Each entity carries a \emph{multiplicity}: the number of original
particles it represents. An initial particle has multiplicity~$1$.
At each coalescence, the heir has the sum of incoming multiplicities
while the ghost has multiplicity~$0$.
\end{definition}

\begin{example}
If a particle with multiplicity~$3$ (arising from earlier mergers)
coalesces with a particle with multiplicity~$2$, the heir has
multiplicity~$5$ and the ghost has multiplicity~$0$.
\end{example}

\subsubsection{Coalescence performance}

A \emph{coalescence performance} specifies:
\begin{itemize}
\item A genealogy forest: which particles merged and where;
\item Ghost paths: where each ghost traveled after emerging.
\end{itemize}
The weight of a performance is the product of all path weights.
(Unlike annihilation, coalescence ghosts are individually
identifiable---each emerges from a specific junction---so no
random labeling factor is needed.)
See~\cite{SU2026coalescence} for the full treatment, including
the coalescence formula and the coalescence determinant.

\subsection{The pairwise coalescence problem}
\label{sec:pfaffian-background}

Given $n = 2k$ particles at initial positions 
$x_1 \preceq \cdots \preceq x_n$, we ask:
what is the total weight (or probability) of the \emph{pairwise coalescence}
event---the event that by time~$T$, each consecutive pair
$\{x_1, x_2\}$, $\{x_3, x_4\}$, $\ldots$, $\{x_{n-1}, x_n\}$ has
coalesced, meaning that the particles within each pair have
merged (directly or through a chain of intermediate collisions)
into a single heir? (Different pairs may merge separately
or may share collisions.)

Computing this from a coalescence formula would require summing over
all coalescence scenarios consistent with the required pairings---different
coalescence trees and ghost configurations. We avoid this difficulty by
converting the problem to annihilation.

\subsection{The coalescence-to-annihilation map}
\label{sec:pfaffian-map}

The classical connection between coalescence and
annihilation via parity---assigning independent random
labels and propagating by symmetric difference---goes
back to Griffeath's cancellative
duality~\cite{Griffeath1979} and was made explicit at
the particle level by ben-Avraham and
Brunet~\cite{benAvrahamBrunet2005} (see also Athreya
and Swart~\cite{AthreyaSwart2012} for a rigorous
treatment). In the ghost framework, the conversion is
deterministic: each entity carries the parity of its
multiplicity, and under this \emph{cancellative labeling},
pairwise coalescence becomes complete annihilation.

\subsubsection{Parity of an entity}

Define the \emph{parity} of an entity as the parity of its
multiplicity (\Cref{def:multiplicity}):
\emph{odd} (representing an odd number of original particles) or
\emph{even} (representing an even number, including zero for ghosts).

\subsubsection{Characterizing pairwise coalescence}

The event ``all $k$ consecutive pairs have coalesced'' is equivalent to
``every final heir has even multiplicity.'' For the forward direction:
each required pair contributes two original particles to the heir that
absorbed it; if all pairs are absorbed, every heir's multiplicity is
even. For the converse, suppose every final heir has even
multiplicity. By the consecutive collision property
(\Cref{def:planar}) particles merge only with neighbors, so each
heir represents a contiguous block of original indices, and even
multiplicity makes every block even in length. The block
boundaries therefore fall at even positions, so no boundary
separates a consecutive pair $\{2i-1, 2i\}$: each such pair lies
in a single heir and has coalesced.

\subsubsection{The cancellative labeling}

Given a coalescence performance where every final heir has even
multiplicity, reclassify each entity as follows:
\begin{itemize}
\item every entity with even multiplicity is reclassified as
  a \emph{ghost};
\item every entity with odd multiplicity remains a \emph{particle}.
\end{itemize}

At each binary coalescence step, the three cases are:

\medskip
\begin{center}
\small
\begin{tabular}{lll}
\hline
\textbf{Incoming parities}
  & \textbf{Outgoing parities}
  & \textbf{Reclassified as} \\
\hline
odd + odd
  & even + zero
  & ghost + ghost (annihilation) \\
odd + even
  & odd + zero
  & particle + ghost (non-event) \\
even + even
  & even + zero
  & ghost + ghost (non-event) \\
\hline
\end{tabular}
\end{center}
\medskip

In the first case, two odd-parity entities (particles in the annihilation
picture) produce two ghosts: this is an annihilation event. In the other
cases, at least one incoming entity has even parity (already a ghost);
the meeting is a non-event. In the third case, both incoming
entities already have even parity (both are ghosts in the
reclassified picture), so their meeting is likewise a
non-event: ghosts pass through each other without interacting.

Since every final heir has even multiplicity, every final heir is
reclassified as a ghost. The result is \emph{complete annihilation}:
zero survivors and $n = 2k$ ghosts. The odd+odd meetings supply
exactly $k$ twin pairs, each born at a collision vertex, so the
reclassified trajectories form a genuine complete-annihilation
performance (\Cref{def:performance}) with $k$ ghost pairs, to
which \Cref{thm:annihilation} applies.

\subsection{The Pfaffian formula}
\label{sec:pfaffian-formula}

We now apply the annihilation formula to derive the Pfaffian structure.

\subsubsection{The antisymmetric matrix}

For particles $I < J$ (with initial positions $x_I \preceq x_J$),
the \emph{pairwise annihilation weight} $A_{IJ}$ is the total
weight of the pairs of trajectories---one from~$x_I$, one
from~$x_J$---that meet. Equivalently (\Cref{sec:intro}), it is
the weight of the event that $I$ and~$J$ annihilate when they
are the only two particles present. It evaluates to
\begin{equation}\label{eq:pairwise}
  A_{IJ} = 2\sum_{\substack{a, b \\ a \prec b}}
  W(x_I \to b) \cdot W(x_J \to a)
  \;+\; \sum_{c}
  W(x_I \to c) \cdot W(x_J \to c),
\end{equation}
where both sums are over final vertices. Two trajectories with
swapped endpoints---particle~$I$, starting on the left, ending
strictly right of particle~$J$---must cross by planarity, and
the first sum collects these; the factor~$2$ accounts for the
crossing pairs whose endpoints are \emph{not} swapped, which
the segment swap beyond the first meeting places in
weight-preserving bijection with the swapped ones. The second
sum is the diagonal correction for the pairs that meet at a
common final vertex~$c$. For continuous state spaces such as
Brownian motion, coincident endpoints have measure zero and
this second sum vanishes; the sums become integrals, and the
results of this section rest on
\Cref{thm:continuous-annihilation}.

Extend to an $n \times n$ antisymmetric matrix by setting
$A_{JI} = -A_{IJ}$ for $I < J$ and $A_{II} = 0$.

Recall the Pfaffian of an antisymmetric
$2k \times 2k$ matrix~$A$:
\begin{equation}\label{eq:pfaffian-def}
\Pf(A) = \sum_{\mu} \sgn(p_1, q_1, \ldots, p_k, q_k)
\prod_{l=1}^{k} A_{p_l q_l},
\end{equation}
where the sum runs over perfect matchings
$\mu = \{\{p_1, q_1\}, \ldots, \{p_k, q_k\}\}$ of
$\{1, \ldots, n\}$, normalized by $p_l < q_l$ and
$p_1 < \cdots < p_k$, and the sign is that of the listed
sequence viewed as a permutation of $(1, \ldots, n)$.

Only the entries above the diagonal enter the
sum~\eqref{eq:pfaffian-def}, but the antisymmetric
extension ties the combinatorial definition to the
classical viewpoint, in which the Pfaffian of an
antisymmetric matrix is a square root of its
determinant: $\Pf(A)^2 = \det(A)$. In particular,
\Cref{thm:pfaffian} below exhibits the total weight of
pairwise coalescence as a square root of a determinant
of pairwise annihilation weights, $w^2 = \det(A)$; the
Pfaffian point processes of \Cref{sec:pfaffian-discussion}
are likewise built from antisymmetric kernels.

\subsubsection{Statement and proof}

\begin{theorem}[Pfaffian pairwise coalescence formula]
\label{thm:pfaffian}
For $n = 2k$ particles at initial positions
$x_1 \preceq \cdots \preceq x_n$, the total
weight of the pairwise coalescence event is
\begin{equation}\label{eq:pfaffian}
  w = \Pf(A).
\end{equation}
\end{theorem}

\noindent
No factor $k!$ appears: the pairwise coalescence event is an
event of the coalescence model, whose performances carry no
ghost-pair numbering; the $k!$ of the annihilation side
arises, and cancels, inside the proof. In the probability
regime, $w$ is the probability of the event.

\begin{proof}
By the coalescence-to-annihilation map (\Cref{sec:pfaffian-map}), the
pairwise coalescence weight~$w$ equals the total weight of complete
annihilation of $n = 2k$ particles. The performances of
\Cref{def:performance} carry a ghost-pair numbering that the
coalescence side does not see: each unnumbered outcome is counted
once per numbering, $k!$ times in total. Hence
\[
k! \cdot w
= \sum_{\FinalState} Z(\FinalState)
= \sum_{\FinalState} \det(M_{\FinalState}),
\]
where the sum runs over all numbered final states: all ghost
positions and the resulting ghost sign patterns. We evaluate the
right-hand side by the annihilation formula
(\Cref{thm:annihilation}) with zero survivors and $k$ ghost pairs.

\smallskip\noindent
\emph{Grouping by matchings.}
Fix a final state and expand $\det(M)$ via the Leibniz formula.
With zero survivors, every bijection
$\pi\colon \Actors \to \Ghosts$ splits the actors into $k$
pairs---the two actors sharing a ghost pair---so the terms that
survive the evaluation of the formal variables group by the
induced perfect
matching $\mu = \{\{p_1, q_1\}, \ldots, \{p_k, q_k\}\}$ of
$\{1, \ldots, n\}$, normalized as in~\eqref{eq:pfaffian-def}. For
a fixed matching, a candidate bijection is determined by the
numbering~$\sigma$ that assigns the block $\{p_l, q_l\}$ to the
ghost pair $j = \sigma(l)$: candidacy (\Cref{def:candidate})
fixes the slot assignment within each ghost pair, so each
matching contributes exactly $k!$ candidate bijections, one per
numbering.

\smallskip\noindent
\emph{The sign is the Pfaffian sign.}
Consider the candidate bijection for matching~$\mu$ and
numbering~$\sigma$. Listing the actors in the role order of
the roles (\Cref{def:role-order}; with $s = 0$ there are no
survivor slots), namely
$(1,1), (1,2), \ldots, (k,1), (k,2)$, produces the sequence whose
$j$-th block is $(q_{\sigma^{-1}(j)}, p_{\sigma^{-1}(j)})$ when
$\varepsilon_j = +1$ (candidacy puts the higher-indexed actor in
slot~$(j,1)$) and $(p_{\sigma^{-1}(j)}, q_{\sigma^{-1}(j)})$ when
$\varepsilon_j = -1$. Hence
\[
\sgn(\pi)
= \sgn(p_1, q_1, \ldots, p_k, q_k) \cdot
  (-1)^{\#\{j :\, \varepsilon_j = +1\}},
\]
since permuting the size-two blocks (which removes~$\sigma$) is
an even permutation and each block reversal is a transposition.
The formal variable relations~\eqref{eq:formal} contribute a
second factor $(-1)^{\#\{j :\, \varepsilon_j = +1\}}$: the
relation yields $-1$ when $\varepsilon_j = +1$ and $+1$
when $\varepsilon_j = -1$. The two sign-pattern-dependent factors
cancel, leaving the Pfaffian sign
$\sgn(p_1, q_1, \ldots, p_k, q_k)$, independent of the ghost
signs and of the numbering.

\smallskip\noindent
\emph{Summing over final states.}
Fix the matching~$\mu$ and sum over all ghost positions, grouped
by the resulting sign pattern. The numbering~$\sigma$ merely
relabels which ghost pair carries which block of~$\mu$, so after
the sum over positions the $k!$ numberings contribute equally,
matching the factor $k!$ on the left-hand side; cancelling it
leaves a single copy per matching. (The cancellation is valid
over every commutative ring: with the edge weights taken as
indeterminates over~$\ZZ$, both sides live in a torsion-free
ring, and the identity specializes.) For that copy the sum
factorizes over the blocks. A block $\{I, J\}$ with $I < J$
occupying ghost pair~$j$
contributes, for $\varepsilon_j = +1$ (positions
$a_j \preceq b_j$, actor~$J$ at the leftward position),
$\sum_{a \preceq b} W(x_I \to b)\, W(x_J \to a)$, and for
$\varepsilon_j = -1$ (positions $b_j \prec a_j$, actor~$I$ at
the rightward position),
$\sum_{b \prec a} W(x_I \to a)\, W(x_J \to b)$. Together:
\[
2\sum_{a \prec b} W(x_I \to b)\, W(x_J \to a)
+ \sum_c W(x_I \to c)\, W(x_J \to c) = A_{IJ}.
\]

Multiplying over the blocks and summing over matchings with the
Pfaffian sign gives, by~\eqref{eq:pfaffian-def},
\[
w = \sum_{\mu} \sgn(p_1, q_1, \ldots, p_k, q_k)
    \prod_{l=1}^{k} A_{p_l q_l}
  = \Pf(A). \qedhere
\]
\end{proof}

\begin{example}[Four particles, two pairs]\label{ex:pfaffian-n4}
Four particles at $x_1 \leq x_2 \leq x_3 \leq x_4$ undergo
pairwise coalescence: pair $\{x_1, x_2\}$ and pair $\{x_3, x_4\}$
each merge by time~$T$ 
(including the case when all four particles coalesce together). 
By the cancellative labeling, this is
equivalent to complete annihilation of all four particles.
The $4 \times 4$ antisymmetric matrix~$A$ has entries
$A_{IJ}$ given by~\eqref{eq:pairwise}, and the Pfaffian
expands as
\[
w = \Pf(A) = A_{12}\, A_{34}
  - A_{13}\, A_{24}
  + A_{14}\, A_{23}.
\]
The three terms correspond to the three perfect matchings of
$\{1, 2, 3, 4\}$. The matching $\{(1,3),(2,4)\}$ (sign
$-1$) is not physically realizable: particles $x_1$ and
$x_3$ cannot meet without $x_2$ intervening first (the
consecutive collision property). The remaining matchings
$\{(1,2),(3,4)\}$ and $\{(1,4),(2,3)\}$ (both sign $+1$)
are realizable---the first when adjacent pairs annihilate
directly, the second when $\{x_2, x_3\}$ annihilate first
and then $\{x_1, x_4\}$ meet. However, each product
$A_{ij}\, A_{kl}$ counts all path configurations where the
two pairs meet, including ones with wrong collision order
or unexpected crossings. The counterterm $A_{13}\, A_{24}$
cancels exactly these spurious contributions, by the
sign-reversing involution.
\end{example}

\subsubsection{Random walks with i.i.d.\ steps}

\Cref{thm:pfaffian} applies to any spacetime graph
satisfying the standing assumptions. A natural special
case is $n = 2k$ independent random walks on~$\ZZ$
with common step probabilities: at each discrete time
step, each particle jumps $+1$ with
probability~$p$ or $-1$ with probability~$q = 1-p$,
independently of all other particles. The starting
positions must share a common parity: two $\pm 1$
walkers of opposite parity always occupy distinct
sublattices, so they can exchange order without ever
meeting, violating the crossing
property~\ref{item:crossing}. We therefore take the
$x_i$ all of the same parity, so that every gap
$x_{i+1} - x_i$ is even.

\begin{corollary}[Biased random walk on~$\ZZ$]
\label{cor:iid-walk}
For $n = 2k$ particles at positions
$x_1 < \cdots < x_n$ on~$\ZZ$, all of the same parity,
each performing
independent $\pm 1$ random walks with jump
probabilities $(p, q)$ for $T$~steps, the total
weight of pairwise coalescence is $\Pf(A)$, where
$A$ is the $n \times n$ antisymmetric matrix with
entries
\begin{equation}\label{eq:pairwise-iid}
A_{IJ} \;=\; 2\!\!\sum_{y_1 < y_2}
  w_T(x_I, y_2)\, w_T(x_J, y_1)
  \;\;+\; \sum_{y} w_T(x_I, y)\, w_T(x_J, y)
\end{equation}
for $I < J$, and the transition weight is
\begin{equation}\label{eq:transition-iid}
w_T(x, y) \;=\;
  \binom{T}{\frac{T + y - x}{2}}\,
  p^{\frac{T + y - x}{2}}\,
  q^{\frac{T - y + x}{2}},
\end{equation}
with $w_T(x, y) = 0$ unless
$\tfrac{T + y - x}{2}$ is a non-negative integer
at most~$T$. Both sums in~\eqref{eq:pairwise-iid}
range over all integers~$y$ (respectively pairs
$y_1 < y_2$) of the same parity as $x_I - T$
and~$x_J - T$.
\end{corollary}

\begin{proof}
The spacetime graph is the path graph
$\{0, 1, \ldots, T\} \times \ZZ$ with edges
$(t, z) \to (t+1, z \pm 1)$ weighted by $p$
and~$q$. The path generating function
$W(x \to y) = w_T(x, y)$ is the binomial transition
probability~\eqref{eq:transition-iid}. Substituting
into~\eqref{eq:pairwise} gives~\eqref{eq:pairwise-iid}.
\end{proof}

The first sum in~\eqref{eq:pairwise-iid} counts
crossing path pairs: particle~$I$ (starting left)
ends strictly right of particle~$J$ (starting right).
The second sum counts pairs ending at the same
position. For $p = q = \tfrac{1}{2}$, the transition
weight simplifies to
$w_T(x,y) = \binom{T}{(T+y-x)/2} / 2^T$.

\subsection{Discussion}
\label{sec:pfaffian-discussion}

\subsubsection{Annihilation proves coalescence}

The proof uses the annihilation formula to establish a result about
coalescence. This interplay between the two models is made
possible by the cancellative labeling.
In the physics literature, Masser and
ben-Avraham~\cite{MasserBenAvraham2001} showed that the
$n$-point correlation functions of annihilation and
coalescence are asymptotically identical at large times;
the cancellative labeling makes this connection exact at
the level of individual configurations. On the coalescence
side, Glinyanaya and
Fomichov~\cite{GlinyanayaFomichov2018} proved a
central limit theorem for the cluster
count with Fano factor (variance-to-mean ratio)
$3 - 2\sqrt{2} \approx 0.172$; a central limit theorem of
the same flavor, for the wall count, is derived in the
companion paper~\cite{Sniady2026pfaffian} from the Pfaffian
formula proved here.

\subsubsection{From determinant to Pfaffian}

The annihilation formula is a single determinant. For complete
annihilation, all columns are ghost-pair columns, and marginalizing
over ghost positions produces a Pfaffian. With survivors, the formula
remains a determinant (survivor columns contribute plain transition
weights). The ghost structure thus interpolates between determinants
(with survivors) and Pfaffians (without).

\subsubsection{Relation to prior Pfaffian formulas}

For Brownian motion, \Cref{thm:pfaffian} recovers a
known result. Theorem~3 of Tribe and
Zaboronski~\cite{TZ2011} expresses the even product
moments of annihilating Brownian motions as Pfaffians
of pairwise moments; specializing to the test function
$g \equiv 0$ yields the complete extinction probability
as $\Pf$ of pairwise extinction probabilities. Garrod,
Poplavskyi, Tribe, and
Zaboronski~\cite{GarrodPTZ2018} proved the analogous
identity for discrete nearest-neighbor walks (Lemma~7
in~\cite{GarrodPTZ2018}). Both proceed by the
generator/ODE argument recalled in
\Cref{sec:pfaffian-point-processes}, in contrast to the
combinatorial route taken here.

In an earlier, separate direction,
Mattera~\cite{Mattera2003} established a bijection
between annihilating random walk configurations and
perfect matchings (dimers) of a planar spacetime
graph, and observed that the remaining particles
form a Pfaffian point process. Mattera's setting is
a single specific lattice (simple random walk on
$\ZZ/n\ZZ$ with synchronous updates),
and his Pfaffian is the Kasteleyn Pfaffian of the
spacetime graph.

\Cref{thm:pfaffian} extends this identity to arbitrary
planar weighted directed acyclic graphs and provides a
combinatorial
explanation: the Pfaffian appears because annihilating
particles pair up in perfect matchings, and
marginalizing over ghost positions factorizes across
pairs.

\subsubsection{Connection to gap statistics}

Via \emph{checkerboard duality}, the companion
paper~\cite{Sniady2026pfaffian} applies the Pfaffian
formula proved here to the walls of a coalescing system,
obtaining empty-interval probabilities and linking the
ghost framework to Pfaffian point process
theory~\cite{TZ2011,GarrodPTZ2018}.

\appendix
\crefalias{section}{appendix}
\crefname{appendix}{Appendix}{Appendices}
\Crefname{appendix}{Appendix}{Appendices}

\section{Prescribed annihilation: computational evidence}
\label{sec:prescribed}

The annihilation formula (\Cref{thm:annihilation}) introduces
ghost particles to restore the $n \times n$ matrix after
annihilation has reduced the number of surviving particles, and
labels the resulting ghost pairs by a uniform random
numbering, since ghosts do not remember which initial
particles produced them
(\Cref{sec:setup-annihilation}). Can the formula
be refined to specify which particles annihilate---that is, can
one obtain a Karlin--McGregor-type expression for the weight of
a prescribed annihilation pattern? For $n = 2$ the question
is vacuous (there is only one pair).

\subsection{Prescribed annihilation}
\label{sec:prescribed-problem}

\begin{definition}[Prescribed annihilation]
\label{def:prescribed}
Given $n$ particles at $x_1 \prec \cdots \prec x_n$, fix
adjacent particles $x_m$ and $x_{m+1}$. The \emph{prescribed
annihilation weight} $Z_{\text{prescribed}}$ is the total weight
of all evolutions (joint realizations of all $n$ particle
trajectories) in which $x_m$ and~$x_{m+1}$ annihilate
(producing a ghost pair at positions $a \prec b$) while the
remaining $n - 2$ particles reach prescribed survivor
positions, without colliding with the annihilating pair
before annihilation.
\end{definition}

Since the ghost pair's origin is determined by the constraint,
the question is whether $Z_{\text{prescribed}}$ can be expressed
as a linear combination of Karlin--McGregor products:
\begin{equation}\label{eq:prescribed-hope}
Z_{\text{prescribed}} =
\sum_{\pi \in S_n} c_\pi \prod_{i=1}^{n}
W(x_i \to z_{\pi(i)}),
\end{equation}
with rational coefficients $c_\pi$ independent of the positions
and of time. Here $z_1, \ldots, z_n$ denote the $n$ final
positions (survivors and ghosts combined).

\subsection{Computation}
\label{sec:prescribed-computation}

For $n = 3$ (the first nontrivial case), on the $\ZZ$ lattice
(simple random walk with $\pm 1$ steps), the prescribed
annihilation weight can be computed exactly by
enumerating all $2^{nt}$ path evolutions and filtering to those
satisfying the constraints of \Cref{def:prescribed}. For a
fixed ordering of the $n$ final positions, distinct
final-position tuples provide linear equations in the $n!$
unknowns~$c_\pi$ of~\eqref{eq:prescribed-hope}. The resulting
system is solved over~$\mathbb{Q}$ in exact arithmetic. (The
same technique was used to discover the annihilation formula of
the present paper and the coalescence formula of the companion
paper~\cite{SU2026coalescence}; it succeeds when a formula
of the form~\eqref{eq:prescribed-hope} exists.)

With initial positions
$x_i = 2(i-1)$ and annihilating pair $(1,2)$, at time $t = 5$ the
ordering $a \prec y_1 \prec b$ yields $10$ distinct tuples
$(a, y_1, b)$:
%
%
\[
\begin{gathered}
(-3, -1, 1), \quad (-3, -1, 3), \quad (-3, -1, 5), \quad
(-3, 1, 3), \quad (-3, 1, 5), \\
(-3, 3, 5), \quad (-1, 1, 3), \quad (-1, 1, 5), \quad
(-1, 3, 5), \quad (1, 3, 5).
\end{gathered}
\]
The resulting $10 \times 6$ system over~$\mathbb{Q}$ is
inconsistent: the coefficient matrix has rank~$6$---the six
products $\prod_i W(x_i \to z_{\pi(i)})$ are linearly
independent as functions of the tuple---while the augmented
matrix has rank~$7$. No rational coefficients~$c_\pi$ satisfy
all ten equations; already some subsystems of six equations are
inconsistent, while every subsystem of five or fewer equations
admits a solution.
The computation, including the tuples, the exact weights, and the
rank certificate, is reproduced by the archived companion
code~\cite{Sniady2026companion}.

\subsection{Interpretation}
\label{sec:prescribed-interpretation}

The inconsistency is not a failure of a specific ansatz such as
a determinant or permanent: every expression of the
form~\eqref{eq:prescribed-hope} is a linear combination of the
$n!$ products $\prod_i W(x_i \to z_{\pi(i)})$, and the linear
system tests all such combinations simultaneously. Moreover, the
computation already grants the weakest version of the ansatz: it
restricts to a single ordering class $a \prec y_1 \prec b$, so
the coefficients are allowed to depend on the ordering of the
final positions, and even a single time horizon $t = 5$ is
used. No rational coefficients~$c_\pi$ exist even under these
relaxations.

Ghost anonymity is not merely a convenient bookkeeping device;
the computation above suggests that it is essential for the
existence of the formula.

\section*{Acknowledgments}

We thank Maciej Dołęga, Marcin Lis, Sho Matsumoto,
and Karol Życzkowski for stimulating discussions and
bibliographic suggestions.

Claude Code (Anthropic) was used as an assistant during formula
discovery and manuscript preparation.

Research supported by grant 2025/59/B/ST1/01258 of the
National Science Centre, Poland.

\section*{Data availability}

No datasets were generated or analysed during the current study;
the results are theoretical. The companion software that numerically
verifies the formulas and reproduces the computations reported in the
paper is openly available in the Zenodo repository,
\url{https://doi.org/10.5281/zenodo.21218341}~\cite{Sniady2026companion}.

\printbibliography

\end{document}